\documentclass[preprint,12pt]{elsarticle}

\usepackage{amsthm} 
\usepackage{natbib}
\usepackage{amsmath}
\usepackage{mathrsfs}
\usepackage{amssymb}
\usepackage{algorithm}
\usepackage{algorithmic}
\usepackage[draft,english]{fixme}
\usepackage{enumerate}

\newcommand{\IM}{\mathrm{Im}\mbox{ } }

\newcommand{\Rank}{\mathrm{rank}\mbox{ } }
\newcommand{\SPAN}{\mathrm{Span}}
\newtheorem{Definition}{Definition}
\newtheorem{Theorem}{Theorem}
\newtheorem{Lemma}{Lemma}

\newtheorem{Proposition}{Proposition}
\newtheorem{Corollary}{Corollary}
\newtheorem{Remark}{Remark}
\newtheorem{Notation}{Notation}
\newtheorem{Example}{Example}

\newtheorem{Procedure}{Procedure}

\newcommand{\LSS}{LSS}
\newcommand{\BLSS}{LSS\ }
\newcommand{\SLSS}{LSSs}
\newcommand{\BSLSS}{LSSs\ }

\newcommand{\MORPH}{\mathcal{S}}
\newcommand{\QNUM}{D}

\newcommand{\Z}{\mathrm{Z}}
\newcommand{\K}{\mathrm{K}}
\newcommand{\SwitchSysLin}[1][{}]{ (n^{#1},Q,\{(A_{q}^{#1},B_{q}^{#1},C_{q}^{#1}) \mid q \in Q \})
}

\newcommand{\Q}{\mathscr{Q}}
\newcommand{\CP}{\mathcal{P}}

\begin{document}


\newcommand{\diag}{\mathrm{diag}}



\begin{frontmatter}

\title{Balanced truncation for linear switched systems}
\author[FIRST]{Mih\'aly Petreczky}
\ead{mihaly.petreczky@mines-douai.fr}
\author[SECOND]{Rafael Wisniewski}
\ead{raf@es.aau.dk}
\author[SECOND]{John Leth\corref{cor1}}
\ead{jjl@es.aau.dk}
\cortext[cor1]{Corresponding author}

\address[FIRST]{Univ Lille Nord de France, F-59000 Lille, France, and
  EMDouai, IA, F-59500 Douai, France
}
\address[SECOND]{
Aalborg University, Dept. of Electronic Systems,
Fr. Bajers Vej 7, C3-211, DK-9220 Aalborg Ost, Denmark
}

\begin{abstract}
In this paper, we present a theoretical analysis of the model
reduction algorithm for linear switched systems from
\cite{Shaker1,Shaker2}.  This algorithm is a reminiscence of the
balanced truncation method for linear parameter varying
systems~\cite{WoodGoddardGlover96}.
Specifically in this paper, we provide a bound on the approximation error
in $L_2$ norm for continuous-time and $l_2$ norm for discrete-time linear switched
systems. We provide a system theoretic interpretation of
grammians and their singular values. Furthermore,  we show that the
performance of balanced truncation depends only on the input-output
map and not on the choice of the state-space representation. 
For a class of stable discrete-time linear switched
systems (so called strongly stable systems), we define nice
controllability and nice observability grammians, which are genuinely
related to reachability and controllability of switched systems.  In
addition, we show that quadratic stability and LMI estimates of the
$L_2$ and $l_2$ gains depend only on the input-output map.

\end{abstract}

\end{frontmatter}

\section{Introduction}

In this paper, we address certain theoretical problems which arise in
model reduction of continuous- and discrete-time Linear Switched
Systems (in this work, referred to as LSS). 

To describe the contribution of this paper in more details, we will
first present an informal overview of balanced truncation for
linear switched system which appeared in \cite{Shaker1,Shaker2}. In
fact, this method has already been used for linear time varying
systems in \cite{WoodGoddardGlover96}.
  
Consider a linear switched system of the form
\begin{equation}
\label{linswitch0}
\Sigma : \left\{\begin{array}{rcl}
\delta(x)(t)
&=& A_{q(t)} x(t) + B_{q(t)} u(t),~ x(t_0) = x_0 \\
y(t) &=& C_{q(t)} x(t).
\end{array}\right.
\end{equation}
where $\delta(x)(t)=\dot x(t)$ in continuous-time case, 
$\delta(x)(t)=x(t+1)$ in discrete-time case, 
$A_{q}$, $B_q$, $C_q$ are $n \times n$, $n \times m$ and $p
\times n$ matrices respectively, and the switching signal $q$ maps
time instances to discrete states in a set $Q$.

We seek to replace the system $\Sigma$ by another one of smaller
dimension, but which still adequately approximates the input-output
behavior of $\Sigma$. The methods for continuous-time and discrete
time LSSs are similar. Specifically, for continuous-time systems, we
define the following \emph{observability grammian} as any positive
definite $\Q > 0$ such that
\begin{equation}
\label{ContObsGramIntro} 
\forall q \in Q: A^T_q\Q+\Q A_q +C_q^TC_q < 0.
\end{equation}
Likewise, define a \emph{controllability grammian} as a
 positive definite $\CP > 0$ that satisfies
\begin{equation}
\label{ContContrGramIntro}
\forall q \in Q: A_q\CP+\CP A_q^T +B_qB^T_q < 0.
\end{equation}
By applying a suitable state-space isomorphism, the system can be
brought into a form where
$\CP=\Q=\Lambda=\diag(\sigma_1,\ldots,\sigma_n)$ are diagonal matrices
and $\sigma_1 \ge \ldots \ge \sigma_n > 0$.  We will call the numbers
$\sigma_1 \ge \ldots \ge \sigma_n > 0$ the \emph{singular values} of
the pair $(\CP,\Q)$. We also observe that
$\sigma_i=\sqrt{\lambda_i(\CP \Q)}$, where $\lambda_1(\CP\Q) \ge
\ldots \ge \lambda_n(\CP\Q)$ are the ordered eigenvalues of $\CP\Q$.
Following the classical terminology, we will call a state-space
representation \emph{balanced}, if $\CP=\Q=\Lambda$, where $\Lambda$
is a diagonal matrix.  We reduce the dimension of a balanced
state-space representation by discarding the last $n-r$ state-space
components.

Suppose that $\Sigma$ is balanced.  Consequently, the system matrices
$\hat{A}_{q}, \hat{B}_q,\hat{C}_q$ of the reduced order system
$\hat{\Sigma}$ are obtained by partitioning the matrices of the
original system as follows
\begin{equation*}
{A}_q=\begin{bmatrix}
\hat{A}_q & \star \\
\star & \star
\end{bmatrix} \mbox{, \ \ } {B}_q = \begin{bmatrix} \hat{B}_q \\
\star
\end{bmatrix} \mbox{, \ \ \ } {C}_q^T = \begin{bmatrix} \hat{C}_q^T, &
\star \end{bmatrix},
\end{equation*}
where $\hat{A}_q,\hat{B}_q,\hat{C}_q$ are $r \times r$, $r \times m$
and $p \times r$ matrices respectively. The performance of this
procedure has been extensively tested by means of numerical examples
in \cite{Shaker1,Shaker2}.  However, many theoretical questions remain
open.

\textbf{Problem formulation} \\
We strive to find error bounds and to establish the invariance of
the method with respect to state-space representation. In particular,
we seek answers to the following questions:
\begin{enumerate}
\item
\label{Problem1}
\textbf{Error bounds} \\
Can we state a bound on the error between the original system $\Sigma$
and the reduced one $\hat{\Sigma}$, using some metric~? In particular,
can we extend the well-known result from the linear case, by proving
that
\begin{equation}
\label{error:bound} 
||Y^{\Sigma}-Y^{\hat{\Sigma}}||_{2} \le 2(\sigma_{r+1}+\cdots+\sigma_{n}),
\end{equation}
where $Y^{\Sigma}$ and $Y^{\hat{\Sigma}}$ are the input-output maps of
$\Sigma$ and $\hat{\Sigma}$ respectively, and $||.||_{2}$ denotes the
$L_2$ norm of the switched system as defined in \cite{HespanhaNov03} ?

\item
\label{Problem4}
\textbf{Invariance of the $L_2$ (for continuous-time) and $l_2$ (for
  discrete-time) norm estimates on state-space representation} \\
Is it possible to estimate the system norm in a manner which does not
depend on the choice of the basis of the state-space ?

\item
\label{Problem2}
\textbf{Invariance of the grammians on state-space representation}\\
Under which conditions the controllability and observability
inequalities \eqref{ContObsGramIntro} and \eqref{ContContrGramIntro}
and their discrete counterparts have solutions ?  Does the existence
of a solution to these inequalities depend on the choice of the
state-space representation ? Can we characterize the set of
observability/controllability grammians in a way which does not depend
on the choice of the state-space representation but only on the
input-output map ?

\item
\label{Problem3}
\textbf{Invariance of the singular values on state-space representation} \\
Do the singular values of the system (i.e., the values
$\sigma_1,\ldots,\sigma_n$) depend only on the input-output map of the
system or do they also depend on the choice of the state-space
representation. Do they have a system theoretic interpretation ?

\item
\label{Problem5}
\textbf{System theoretic interpretation of the grammians} \\
What is the relationship between grammians and
observability~/~controllability of the switched systems.
In the linear case, existence of strictly positive
observability/controllability grammians implies
observability/controllability of the system. Does this extend to the
switched case ?

\item
\label{Problem7}
\textbf{Preservation of system theoretic properties by
  the reduced system} \\
If the original system was reachable, observable, minimal, stable,
etc., then will balanced truncation preserve these properties ?
       
\end{enumerate}
The motivation for the first problem is clear.  The motivation for
questions \ref{Problem4}--\ref{Problem3} is the following. The
formulation of balanced truncation does not a priory exclude the
possibility that the choice of the state-space representation might
influence the error bound.  This would inhibit the applicability of
the method, since the choice of the initial state-space representation
is often circumstantial.  In particular, the existence of a solution
to the LMIs that are used for estimating $L_2$ and $l_2$ gains might
depend on the choice of the state-space representation.  In a similar
manner, the existence of grammians and the values of the corresponding
singular values might also depend on the state-space representation.
Carrying on, Question \ref{Problem5} is important for obtaining a
deeper theoretical insight and for answering Question \ref{Problem7}.
Whereas, Question \ref{Problem7} is important, because the reduced
system is to be used for control design, which is easier if
certain important system-theoretic properties remain valid.

\textbf{Contribution of the paper}

In this paper, we prove the following results:
\begin{itemize}
\item We prove the error-bound \eqref{error:bound} for continuous and
discrete LSSs.

\item We show how to estimate the $L_2$ and $l_2$ norms using LMIs in
such a way that the obtained estimates do not depend on the choice of
the state-space representation.

\item If a system admits an observability (controllability) grammian,
then any minimal linear switched system which describes the same
input-output map will admit an observability (controllability)
grammian.

There is a one-to-one correspondence between controllability
(resp. observability) grammians of minimal systems which describe the
same input-output map. This correspondence preserves the singular
values.  That is, the existence of grammians is a property of
  the input-output map and not of the state-space representation.
For minimal state-space representations, the singular values are
  functions of the input-output map and not of the state-space
  representation.  We also relate the largest singular value to the
Hankel-norm of the system.

As a byproduct, we also show that if an input-output map can be
  realized by a quadratically stable system\footnote{i.e. a system
    with a common quadratic Lyapunov function}, then any minimal
  realization of this map will be quadratically stable.

\item For minimal systems, if controllability and observability
grammians exist, then they are necessarily strictly positive definite. 

For a class of discrete-time LSS, so called strongly stable
LSSs, for which the matrix $\sum_{q \in Q} A_q^T \otimes A^T_q$ has all
its eigenvalues inside the open unit disc, the converse result also holds. Specifically, if nice controllability and observability grammians are positive definite then the system is minimal. 

\item Balanced truncation preserves quadratic stability. However, it
does not necessarily preserve minimality. The fact that balanced
truncation does not preserve minimality is a further indication of
that the method might be very conservative.

\end{itemize}

\textbf{Related work}
The current paper is an extension of \cite{ADHS2012}. With respect to
this work, the main differences are: 
\begin{enumerate}
\item The current paper contains detailed proofs.
\item It presents results for both the continuous- and discrete-time LSSs.
\item It provides a detailed exposition of the comparison between
balanced truncation of switched systems and balanced truncation of
linear parameter~-~varying systems and systems with structured
uncertainty.
\end{enumerate}

A rich literature covers the subject of model reduction for switched
systems, \cite{French1,Zhang20082944, Mazzi1,
  Habets1,China2,China3,Lam1,Kotsalis2,Kotsalis1,Kotsalis,6209392}.  
In particular, balanced
truncation was explored in
\cite{Shaker1,Shaker2,Kotsalis1,Kotsalis2,Kotsalis,Lam1,6209392}. The procedure dealt
with in this paper was already described in
\cite{Shaker1,Shaker2}. The error bounds of balanced truncation were
investigated in~\cite{WoodGoddardGlover96} for linear parameter
varying systems, in \cite{BeckThesis} for uncertain systems, and
in~\cite{Kotsalis1} for discrete-time jump linear stochastic systems.
The proofs of the error bound in this paper are original and are
inspired by similar proofs for time-varying systems
\cite{SandbergThesis}; yet, by introducing certain constructions, we
show that the results in this paper can be related to the above works.
We will present a detailed comparison between the contribution of this
paper and \cite{WoodGoddardGlover96,BeckThesis,Kotsalis1} in \S
\ref{RelationSection}. In short, the main differences are as
follows. First, this paper deals explicitly with linear switched
systems and does so for both discrete-time and continuous-time.
As a consequence, the obtained results are less conservative than the
ones of \cite{WoodGoddardGlover96,BeckThesis,Kotsalis1} when applied
to linear switched systems.  Second, we address in great detail the
independence of balanced truncation from the choice of state-space
representation.

The induced $L_2$ norm for switched systems was introduced and
analyzed in
\cite{HespanhaNov03,MargaliotHespanhaFeb07,HirataHespanhaDec09}. Nonetheless,
the focus in those works was not on the invariance of the computed
estimates with respect to the choice of the state-space
representation, addressed in this paper.

Equivalent formulations of the concept of strong stability were used in 
\cite{PetreczkyV07,MP:MTNSSpaces}, but in a completely different setting.
In \cite{MP:MTNSSpaces} it was also stated that minimization preserves
strong stability, but no proof was provided.


We believe that the paper, as a whole, represents a new contribution,
although some of the results of this paper have appeared in other
contexts in literature.  It provides a coherent exposition of the
theoretical aspects of balanced truncation, and it does it in a
self-contained manner. The presentation is tailored to linear switched
systems.  Finally, we believe that the relationship between the
approach presented in the current work and balanced truncation methods
of \cite{WoodGoddardGlover96,BeckThesis,Kotsalis1} is also interesting
on its own right, since the corresponding system classes are rather
far from linear switched systems.

\textbf{Outline:} In \S \ref{sect:sys}, we present the formal definition of continuous-
and discrete-time LSSs. Furthermore, we give a brief overview of
realization theory of LSSs.  In \S \ref{sect:stab}, we state the
formal definition of $L_2$ and $l_2$ norms, and
grammians. Subsequently, we show the relationship between these
concepts and conditions which guarantee their existence. In \S
\ref{sect:inv}, we show that quadratic stability, estimates of the
$L_2$ and $l_2$ norms, existence of grammians and the singular values
of the system are all independent on the state-space representation.
In \S \ref{sys:interp}, we discuss the system theoretic interpretation
of grammians and their singular values. For the class of strongly
stable discrete-time LSS, we show that the balanced truncation
conserves strong stability and minimality of the representation.  In
\S \ref{sect:model_red}, we present the proof of the error bound for
balanced truncation. Finally in \S \ref{RelationSection}, we show how our
results are related to the results on balanced truncation of
uncertain systems, linear parameter varying systems and discrete-time
Markov jump linear systems.

\textbf{Notation:} The cardinality of a set $S$ is denoted $|S|$. By
$\mathbb{N}$, we denote the set of natural numbers including $0$, and
by $T=\mathbb R_{+}$ the set of non negative reals.  We denote by
$||x||_{2}$ the Euclidean norm of a vector $x \in \mathbb{R}^{n}$. We
denote by $\mathbb{R}^{k \times l}$ the set of all $k \times l$
matrices with real entries. For a symmetric matrix $A \in
\mathbb{R}^{n \times n}$, we write $A>0$, $A<0$, $A\ge 0$ and $A \le
0$ if $A$ is strictly positive definite, strictly negative definite,
positive semi-definite and negative semi-definite, respectively.  By
$\diag(a_1,a_2,\ldots,a_n)$, we denote the $n \times n$ diagonal
matrix with diagonal entries $a_1,\ldots,a_n \in \mathbb{R}$.
Similarly, we denote by $\diag(A_1,\ldots,A_k)$ the $n \times n$
block-diagonal matrix, such that $A_i \in \mathbb{R}^{n_i \times n_i}$
and $n=\sum_{i=1}^{k} n_i$.

We use the standard notation of automata theory \cite{AutoEilen}.  For
a set $X$, called the \emph{alphabet}, we denote by $X^{*}$ the set of
finite strings of elements of $X$, and by $X^{\omega}$ the set of
infinite strings of elements of $X$. Elements of $X^{*}$ and
$X^{\omega}$ are called finite and infinite words, respectively. 
That is, a typical element of $X^{*}$ is a finite sequence
$w_0\cdots w_k$, $w_0,\ldots,w_k \in X$
and a typical element of $X^{\omega}$ is an infinite
sequence $w_0w_1\cdots $, such that $w_0,w_1,\ldots \in X$.

 For a
finite word $w=w_0w_1\cdots w_k \in X^*$, its length is denoted by
$|w|$ and is equal $k+1$.  We denote by $\epsilon$ the empty sequence
(word) in $X^*$. In addition, we define $X^{+}=X^{*}\setminus
\{\epsilon\}$.

We say that a map $f:T \rightarrow \mathbb{R}^{n}$ is
\emph{piecewise-continuous}, if $f$ has finitely many points of
discontinuity on any compact subinterval of $T$, and at any point of
discontinuity the left-hand and right-hand side limits of $f$ exist
and are finite.  We denote by $PC(T,\mathbb{R}^n)$ the set of all
piecewise-continuous functions of the above form. The notation $AC(T,
\mathbb R^n)$ designates the set of all absolutely continuous maps $f:
T \to \mathbb R^n$. We denote by $L_2=L_{2}(T, \mathbb R^n)$ the set
of all Lebesgue measurable maps $f:T \rightarrow \mathbb{R}^{n}$ for
which $\int_{0}^{\infty} ||f(s)||_{2}^{2} ds < \infty$.  For $f \in
L_2$, we denote by $||f||_{2}$ the standard norm of $f$, i.e.,
$||f||_2=\sqrt{\int_{0}^{\infty} ||f(s)||_{2}^{2} ds}$. Likewise, we
denote by $l_2=l_2(\mathbb R^n)=l_2(\mathbb{N}, \mathbb R^n)$ the set
of all sequences $x=(x_0,x_1,\dots)$ in $\mathbb{R}^n$ with bounded
$l_2$ norm, i.e., $||x||_2=\sqrt{\sum_{i=0}^{\infty}
  ||x_i||_{2}^{2}}<\infty$.  Finally, if $M$ is a $k \times l$ matrix,
then $||M||_{F}$ denotes the Frobenius norm of $M$, i.e.,
$||M||_F=\sqrt{\sum_{i=1}^{k}\sum_{j=1}^{l} M_{i,j}^2}$.  For $M$ a
square matrix, i.e., $k=l$, we denote by $\mathrm{tr}(M)$ the trace of
$M$, $\mathrm{tr}(M)=\sum_{i=1}^{k} M_{i,i}$.

If $z:\mathbb{N} \rightarrow \mathbb{R}^r$ for some $r$, then denote
by $\delta(z)$ the forward shift operator $\delta(z)(t)=z(t+1)$,
$\forall t \in \mathbb{N}$.  If $z \in AC(T,\mathbb{R}^{r})$, then
denote by $\delta(z)$ the derivative operator,
i.e. $\delta(z)(t)=\frac{dz}{dt}(t)$.

\section{Linear switched systems}
\label{sect:sys}
Below, we formulate the definition of (continuous and discrete) linear
switched systems and their system theoretic properties. The
presentation is based on \cite{MP:BigArticle,MP:Phd}.

\begin{Definition}[Linear switched systems]
A \emph{linear switched system} with external switching (abbreviated
as \LSS) is a tuple
\[ \Sigma = \SwitchSysLin, \] where $Q=\{1,\ldots,\QNUM\}$ for some
fixed $\QNUM \in \mathbb{N}\setminus\{0\}$, and for each $q \in Q$,
$(A_q, B_q, C_q) \in \mathbb{R}^{n \times n} \times \mathbb{R}^{n
  \times m} \times \mathbb{R}^{p \times n}$. The number
$n\in\mathbb{N}$, sometimes denoted $\dim \Sigma$, is called the
\emph{dimension} of the \BLSS $\Sigma$. The elements of the set $Q$
will be called the discrete modes, and $Q$ will be called the set of
discrete modes.
\end{Definition}

The phrase ``external switching'' will be explained after we have
introduced the notion of a solution.
 
In the sequel, we use the following notation and terminology: the
state space $X = \mathbb R^n$, the output space $Y = \mathbb R^p$, and
the input space $U = \mathbb R^m$. Moreover, to unify notation, we
write $\mathcal{U}$, $\mathcal{Q}$, $\mathcal{X}$ and $\mathcal{Y}$ to
denote either $L_2(T,U)$, $PC(T,Q)$, $AC(T,X)$ and $PC(T,Y)$ (in
continuous-time) or 
%
%
%
$l_2(U)$, $Q^{\omega}$, $X^{\omega}$ and $Y^{\omega}$ (in discrete
time).  

Occasionally, we write $q(t)$ for the $t$th element $q_t$ of a
sequence $q\in Q^{\omega}$. The same comment applies to the elements of $U^{\omega}$,
$X^{\omega}$ and $Y^{\omega}$.

With the conventions just stated, we collectively denote by
$||\cdot||_2$ either the $l_2$ norm (in discrete-time) or the $L_2$
norm (in continuous-time).


\begin{Definition}[Solution]
A \emph{solution} of the switched system (with external switching)
$\Sigma$ with initial state $x_0 \in X$ and relative to the pair
$(u,q) \in \mathcal{U}\times\mathcal{Q}$ is by definition a pair $(x, y) \in
\mathcal{X}\times\mathcal{Y}$ satisfying
\begin{equation}\label{eq:cs}
\begin{split}
\delta x(t) 
&= A_{q(t)} x(t) + B_{q(t)} u(t)~\text{a.e},~ x(0) = x_0 \\
y(t) &= C_{q(t)} x(t),
\end{split}
\end{equation}
with a.e meaning almost everywhere and $\delta$ denoting the
derivative operator in continuous-time, and the forward shift operator
in discrete-time.
\end{Definition}

In both continuous and discrete-time, we shall call $u$ the control
input, $q$ the switching signal, $x$ the state trajectory, and $y$ the
output trajectory.

Note that the pair $(u,q)\in\mathcal{U}\times\mathcal{Q}$ can be considered as an
input to the \LSS. The phrase ``external input'' in the definition of
an \BLSS refers to the fact that $(u,q)$ can be chosen
externally. Contrary to the situation when $q$ is state-dependent and
the value of $q$ is assigned internally; for instance, when the state
space is partitioned and a specific value of $q$ is assigned for each
cell of the partition.

In the sequel, we only specify whether we consider continuous or
discrete-time when this is not clear from the context. Remarkably,
the particular case is frequently immaterial for reaching our
conclusions.

\begin{Definition}[Input-state and input-output maps]
The \emph{input}-\emph{state} map $X_{x_0}^{\Sigma}$ and
\emph{input-output} map $Y_{x_0}^{\Sigma}$ for the \BLSS $\Sigma$,
induced by the initial state $x_0 \in X$, are the maps
\begin{align*}
\mathcal{U}\times\mathcal{Q}\to\mathcal{X};~
(u,q)\mapsto X_{x_0}^{\Sigma}(u,q)=x,\\
\mathcal{U}\times\mathcal{Q}\to\mathcal{Y};~ (u,q)\mapsto Y_{x_0}^{\Sigma}(u,q)=y,
\end{align*}
where $(x,y)$ is the solution of $\Sigma$ at $x_0$ relative to
$(u,q)$.
\end{Definition}

A natural question is when a map $f$ of the form
\begin{align}\label{iomap}
f:~\mathcal{U}\times\mathcal{Q}\to\mathcal{Y}
\end{align}
is indeed realized by an LSS. In the sequel, we refer to a map
$\eqref{iomap}$ as an input-output map.
To address the above problem, we will fix a designated initial state
for \SLSS. Since we will mostly deal with exponentially stable \SLSS,
we will set the initial state to be zero.\label{zero} While this choice seems
natural for the current paper, other choices might be more appropriate
in other circumstances.  Nonetheless, many of the results of this
paper can be extended to the case of non-zero initial conditions.
\begin{Definition}[Realization and equivalence]
\label{def:real}
The \emph{input}-\emph{output} \emph{map} $Y^{\Sigma}$ of an \BLSS
$\Sigma$ is the input-output map $Y^{\Sigma} = Y^{\Sigma}_{0}$ induced
by the zero initial state.  The \BLSS $\Sigma$ is said to be a
\emph{realization} of an input-output map $f$ of the form
\eqref{iomap}, if $Y^{\Sigma}=f$. Moreover, the \BSLSS $\Sigma_1$ and
$\Sigma_2$ are \emph{equivalent} if $Y^{\Sigma_1}=Y^{\Sigma_2}$.
\end{Definition}

It is clear that any \BLSS is a realization of its own input-output
map induced by the zero initial state.

\begin{Definition}[Minimality]
The \BLSS $\Sigma_{\mathrm m}$ is said to be a \emph{minimal}
realization of an input-output map $f$, if $\Sigma_{\mathrm m}$ is a
realization of $f$ and if for any other \BLSS $\Sigma$ which is a
realization of $f$, $\dim \Sigma_{\mathrm m} \le \dim \Sigma$.  We say
that $\Sigma_{\mathrm m}$ is a minimal \BLSS, if it is a minimal
realization of its input-output map $Y^{\Sigma_{\mathrm m}}$.
\end{Definition}
\begin{Definition}[Observability]
\label{sect:switch_sys:obs:def}
An \BLSS $\Sigma$ is said to be \emph{observable}, if for any two
states $x_{1}\neq x_{2} \in X$, the input-output maps induced by $x_1$
and $x_2$ are different, i.e., $Y^{\Sigma}_{x_1} \ne
Y^{\Sigma}_{x_2}$.
\end{Definition}

Let $Reach_{x_0}(\Sigma)\subseteq X$ denote the reachable set of the
\BLSS $\Sigma$ relative to the initial condition $x_0\in X$, i.e.,
$Reach_{x_0}(\Sigma)$ is the image of the map $(u,q,t)\mapsto
X^{\Sigma}_{x_0}(u,q)(t)$.

\begin{Definition}[(Span-)Reachability]
\label{sect:switch_sys:reach:def}
The \BLSS $\Sigma$ is said to be \emph{reachable} if every state is
reachable from the zero initial state, i.e., if
$Reach_{0}(\Sigma)=X$. The \BLSS $\Sigma$ is \emph{span-reachable} if
$X$ is the smallest vector space containing $Reach(\Sigma)$.
\end{Definition}

We note that span-reachability and reachability are the same in
continuous-time.

As we now recall (in Theorem~\ref{Theo:min} below), there is a strong
relation between minimality on one side and span-reachability and
observability on the other.

\begin{Definition}{(Isomorphism)}
\label{sect:problem_form:lin:morphism}
Two \BLSS
\[\Sigma_1 = \SwitchSysLin,\]
and
\[ \Sigma_{2}=(n, Q, \{(A_q^a,B_{q}^a,C_q^a) \mid q \in Q\}) \] are
said to be \emph{isomorphic} if there exists a non-singular matrix $S
\in \mathbb{R}^{n \times n}$ such that
\begin{equation*}
\forall q \in Q: 
A^{a}_{q}S=SA_{q},  B_{q}^{a}=SB_{q}, 
C_{q}^{a}S=C_{q}.
\end{equation*}
The matrix $S$ is said to be an \emph{isomorphism} from $\Sigma_{1}$
to $\Sigma_{2}$ and is denoted by $S:\Sigma_{1} \rightarrow
\Sigma_{2}$.
\end{Definition}

The following theorem summaries various results on minimality, see
\cite{MPLBJH:Real,MP:RealForm,MP:BigArticle}.

\begin{Theorem}[Minimality]
\label{Theo:min}
A \BLSS realization is minimal, if and only if it is span-reachable
and observable. All minimal \BLSS realizations of an input-output map
are isomorphic. Finally, if $\Sigma_1$ and $\Sigma_2$ are two
equivalent and minimal \SLSS, then they are related by an \LSS\
isomorphism.
\end{Theorem}
\begin{Remark}
 In \cite{MP:BigArticle,MPLBJH:Real,MP:RealForm} a slightly different
 definition of input-output maps was used. There, an input-output map
 $Y_{\Sigma,x_0}$ of $\Sigma$ induced by the initial state $x_0$ was a map of the form
 $Y_{\Sigma,x_0}:PC(T,U) \times (Q \times T)^{+} \rightarrow Y$ for
 the continuous-time case and $Y_{\Sigma,x_0}:(U \times Q)^{+} \rightarrow Y$ 
 for the discrete-time case. 
 The relationship between $Y_{\Sigma,x_0}$ and $Y^{\Sigma}_{x_0}$  is as follows.

 For the continuous-time case, consider a sequence $w=(q_1,t_1)\cdots (q_k,t_k) \in (Q \times T)^{+}$.
 The interpretation of $w$ is as follows: the discrete mode $q_i$ is active for duration
 $t_i$. 
 Assume that $w$ has the property that
 there exists no $i=1,\ldots,k-1$ such that $t_i=t_{i+1}=0$.
 It is clear from \cite{MP:BigArticle,MP:RealForm} that $Y_{\Sigma,x_0}$ is
 uniquely determined by its restriction to the switching sequences satisfying this property.
 Any such switching sequence can be 
 interpreted as a restriction of a signal $q \in \mathcal{Q}$ to the interval 
 $[0,\sum_{i=1}^{k} t_i]$ such that
 $q|_{(\sum_{j=1}^{i-1} t_{j},\sum_{j=1}^{i} t_j]}=q_i$, $q(0)=q_0$, $i=1,\ldots,k$.
 Conversely, for any signal $q \in Q$, the restriction of $q$ to a finite time interval
 can be interpreted as such a switching sequence.
 From the definition of $Y_{\Sigma,x_0}$ presented in
 \cite{MP:BigArticle,MP:RealForm}, it follows that $Y_{\Sigma,x_0}(u,w)$ depends only
 on the restriction of $u$ to 
 $[0,\sum_{i=1}^{k} t_i]$. It is easy to see that for any $u \in PC(T,U)$, there exists
 a $\widetilde{u} \in \mathcal{U}$ such that on $[0,\sum_{i=1}^{k} t_i]$ $u$ and
 $\widetilde{u}$ coincide.
As a consequence, $Y_{\Sigma,x_0}(u,w)=Y^{\Sigma}_{x_0}(\widetilde{u},q)(\sum_{i=1}^{k} t_i)$.
 Hence, there is a one-to-one correspondence between $Y_{\Sigma,x_0}$ and
 $Y^{\Sigma}_{x_0}$.

 For the discrete-time case, 
 $Y_{\Sigma,x_0}((u_0,q_0)\cdots (u_t,q_t))=Y^{\Sigma}_{x_0}(\widetilde{u},q)(t)$, 
 where $(\widetilde{u},v)$ is any element of $\mathcal{U} \times \mathcal{Q}$ such that
 $\widetilde{u}(i)=u_i$ and $v_i=q_i$ for all $i=0,\ldots, t$. Since 
 any element of $(U \times Q)^{+}$ arises from some element of
 $\mathcal{U} \times Q$ in this way,
 it follows that
 there is one-to-one correspondence between $Y_{\Sigma,x_0}$ and $Y^{\Sigma}_{x_0}$.
 
 It then follows that two \BSLSS $\Sigma_1$ and $\Sigma_2$ are equivalent according to
 Definition \ref{def:real} if and only if $Y_{\Sigma_1,0}=Y_{\Sigma_2,0}$, i.e. 
 $\Sigma_1$ and $\Sigma_2$ realize the same input-output map according to
 \cite{MP:BigArticle,MPLBJH:Real,MP:RealForm}. From the discussion above, it is also
 easy to see that the definitions of observability and span-reachability and minimality
 presented above coincide with those in \cite{MP:BigArticle,MPLBJH:Real,MP:RealForm}.
 This means that the results of \cite{MP:BigArticle,MPLBJH:Real,MP:RealForm}  
 on minimal realization can indeed be used, despite the slight difference in the
 definition of input-output maps.
\end{Remark}
\begin{Remark}
Note that it is also possible to extend the notion of a Hankel-matrix
to switched systems (see \S \ref{SS:SV}) and show that existence of a
realization is equivalent to the finiteness of the rank of the
Hankel-matrix. In fact, the rank of the Hankel-matrix will give the
dimension of a minimal realization. For the purposes of this paper
these results are not directly relevant, the interested reader is
referred to \cite{MP:BigArticle,MPLBJH:Real,MP:RealForm}.
\end{Remark}
Observability and span-reachability of an \BLSS can be characterized by
linear-algebraic conditions.  To present these conditions, we
introduce the following notation.
\begin{Notation}
Consider an \BLSS $\Sigma = \SwitchSysLin$.  For a sequence $q \in
Q^{*}$, we write
\[ A_{q} = \begin{cases} I_{n} & \hbox{ if } q=\epsilon,\\
A_{q_k} \cdots A_{q_2} A_{q_1} & \hbox{ if } q=q_1 q_2 \cdots q_{k},
\end{cases}
\]
where $I_{n}$ denotes the $n \times n$ identity matrix; and recall
that $\epsilon$ is the empty sequence.
\end{Notation}

Let $Q_n^*$ be the set of all words $w \in Q^*$ of length at most $n$,
\[ Q_n^* = \{w \in Q^{*} \mid |w| \le n\}.\] Furthermore, denote by
$M$ the cardinality of $Q_n^*$ ($M = |Q_n^*|$). Fix an ordering
$\{v_1,\ldots,v_M\}$ of the set $Q_n^*$. The next theorem is due to
\cite{Sun:Book,MP:Phd,MPLBJH:Real}.
  
\begin{Theorem}
\label{sect:problem_form:reachobs:prop1}
\label{sect:real:lemma1}~\\
\textbf{Span-Reachability}: The \BLSS $\Sigma$ is span-reachable if
and only if \( \Rank \mathcal{R}(\Sigma)=n, \) where
\begin{equation*}
\begin{split}
& \mathcal{R}(\Sigma)=\begin{bmatrix} A_{v_1}\widetilde{B}, &
A_{v_2}\widetilde{B}, & \ldots, & A_{v_M}\widetilde{B}
\end{bmatrix} \in \mathbb{R}^{n \times m\QNUM M}
\end{split}
\end{equation*}
with $\widetilde{B}=\begin{bmatrix} B_1, & B_2, & \ldots,&
B_{\QNUM} \end{bmatrix} \in \mathbb{R}^{n \times
  \QNUM m}$.\\
      
\textbf{Observability}: The LSS $\Sigma$ is observable if and only if
\( \Rank \mathcal{O}(\Sigma)=n, \) where
\[ \mathcal{O}(\Sigma)=\begin{bmatrix} (\widetilde{C}A_{v_1})^{T}, &
(\widetilde{C}A_{v_2})^T, & \ldots, &
(\widetilde{C}A_{v_M})^T \end{bmatrix}^{T} \in \mathbb{R}^{p\QNUM M
  \times n}.
\]
where $\widetilde{C}=\begin{bmatrix} C_1^T & C_2^T, & \ldots, &
C_{\QNUM}^T \end{bmatrix}^T \in \mathbb{R}^{p\QNUM \times n}$.
\end{Theorem}

The matrix $\mathcal{R}(\Sigma)$ (resp. $\mathcal{O}(\Sigma)$) will be
called a \emph{span-reachability matrix} (resp. \emph{observability
  matrix}) of $\Sigma$.

\begin{Remark}
\label{lin:rem1}
If a linear subsystem of an \BLSS $\Sigma$ is observable (reachable),
then $\Sigma$ is observable (resp. span-reachable). Hence, by Theorem
\ref{Theo:min}, if a linear subsystem of $\Sigma$ is minimal, then
$\Sigma$ itself is minimal.
\end{Remark}
\begin{Remark}
\label{lin:rem1.5}
In \cite{Sun:Book,MP:Phd,MPLBJH:Real} it is shown that $\ker
\mathcal{O}(\Sigma)=\bigcap_{q \in Q, v \in Q^{*}} \ker C_qA_v$ and
$\IM \mathcal{R}(\Sigma)=\SPAN\{ A_vB_qu \mid v \in Q^{*}, q \in Q, u
\in \mathbb{R}^m\}$.
\end{Remark}
\begin{Remark}
\label{lin:rem2}
Note that observability (span-reachability) of an \BLSS does not imply
observability (reachability) of any of its linear subsystems. In fact,
it is easy to construct a counter example \cite{MP:BigArticle}.
Together with Theorem \ref{Theo:min}, which states that minimal
realizations are unique up to isomorphism, this implies that
\emph{there exists an input-output map which can be realized by an
  \LSS, but which cannot be realized by an \BLSS where all (or some)
  of the linear subsystems are observable (or reachable).}
\end{Remark}
\begin{Remark}[Duality between span-reachability and observability]
\label{rem:dual}
 If $\Sigma=\SwitchSysLin$, then define  the \emph{dual system of $\Sigma$} as 
 $\Sigma^T=(n,Q,\{(A_q^T,C_q^T,B_q^T) \mid q \in Q\})$.  From Theorem \ref{sect:real:lemma1}
 it then follows that $\Sigma$ is observable if and only if $\Sigma^{T}$ is span-reachable.
 Conversely, $\Sigma$ is span-reachable, if and only if $\Sigma^T$ is observable.
 That is, similarly to linear systems, reachability and observability are dual 
 properties for \SLSS.
\end{Remark}

From \cite{MP:BigArticle,MP:Phd,MPLBJH:Real}, we have the following algorithms for
reachability and observability reduction, and minimal representation.

\begin{Procedure}
\label{LSSreach}
~ \\
\textbf{Reachability reduction}:
Assume that $\Rank \mathcal{R}(\Sigma)= r$ and choose a basis
$b_1,\ldots,b_n$ of $\mathbb{R}^n$ such that $b_1,\ldots,b_{r}$ span
$\IM \mathcal{R}(\Sigma)$.  In the new basis, the matrices
$A_q,B_q,C_q$, $q \in Q$ become as follows
\begin{equation}
\label{LSSreach:eq1} 
A_{q}=\begin{bmatrix} A_{q}^{\mathrm R} & A^{'}_{q} \\ 0 & A^{''}_{q} \end{bmatrix},
C_{q}=\begin{bmatrix} C_q^{\mathrm R}, & C_{q}^{'} \end{bmatrix},
B_{q}=\begin{bmatrix} B_{q}^{\mathrm R} \\ 0 \end{bmatrix},
\end{equation}
where $A^{\mathrm R}_q \in \mathbb{R}^{r \times r}, B_q^{\mathrm R}
\in \mathbb{R}^{r \times m}$, and $C^{\mathrm R}_q \in \mathbb{R}^{p
  \times r}$.  As a consequence, $\Sigma^{\mathrm R}=(r, Q,
\{(A_q^{\mathrm R}, B_q^{\mathrm R}, C_q^{\mathrm R})~ |~q \in Q\})$
is span-reachable, and has the same input-output map as $\Sigma$.
\end{Procedure}
Intuitively, $\Sigma^{\mathrm R}$ is obtained from $\Sigma$ by
restricting the dynamics and the output map of $\Sigma$ to the
subspace $\IM \mathcal{R}(\Sigma)$.

\begin{Procedure}
\label{LSSobs}
~ \\
\textbf{Observability reduction}:
Assume that $\ker \mathcal{O}(\Sigma)=n-o$, and let $b_1,\ldots,b_n$
be a basis in $\mathbb{R}^n$ such that $b_{o+1},\ldots,b_{n}$ span
$\ker \mathcal O(\Sigma)$.  In this new basis, $A_q$, $B_q$, and $C_q$
can be rewritten as
\[ A_{q}=\begin{bmatrix} A_{q}^{\mathrm O} & 0 \\ A^{'}_{q} &
A^{''}_{q} \end{bmatrix}, C_{q}=\begin{bmatrix} C_q^{\mathrm O}, &
0 \end{bmatrix}, B_{q}=\begin{bmatrix} B_{q}^{\mathrm O} \\
B_{q}^{'} \end{bmatrix},
\]
where $A^{\mathrm O}_{q} \in \mathbb{R}^{o \times o}, B_q^{\mathrm O}
\in \mathbb{R}^{o \times m}$, and $C_q^{\mathrm O} \in \mathbb{R}^{p
  \times o}$.  Consequently, the \BLSS $\Sigma^{\mathrm O}=(o, Q,
\{(A_q^{\mathrm O}, B_q^{\mathrm O}, C_q^{\mathrm O})~ |~q \in Q\})$
is observable and its input-output map is the same as that of
$\Sigma$. If $\Sigma$ is span-reachable, then so is $\Sigma^{\mathrm O}$.
\end{Procedure}
Intuitively, $\Sigma^{\mathrm O}$ is obtained from $\Sigma$ by merging
any two states $x_1$, $x_2$ of $\Sigma$, for which
$\mathcal{O}(\Sigma)x_1=\mathcal{O}(\Sigma)x_2$.

\begin{Procedure}
~ \\
\textbf{Minimal representation}:\\
\label{LSSmin}
Transform $\Sigma$ to a reachable \BLSS $\Sigma^{\mathrm R}$ by
Procedure~\ref{LSSreach}. Subsequently, transform $\Sigma^{\mathrm R}$
to an observable \BLSS $\Sigma^{\mathrm M}=(\Sigma^{\mathrm
  R})^{\mathrm O}$ using Procedure~\ref{LSSobs}.  Then
$\Sigma^{\mathrm M}$ is a minimal \BLSS which is equivalent to
$\Sigma$.
\end{Procedure}

\section{Stability, grammians and $\mathcal{L}_2$ norms}
\label{sect:stab}
In this section, we briefly review the definition of
controllability/observability grammians, $\mathcal{L}_2$ norm, and
quadratic stability for \SLSS. We also recall the basic relationships
between these concepts.
\begin{Definition}[Quadratic stability]
An \BLSS\ \[\Sigma = \SwitchSysLin[{}]\] is said to be quadratically
stable if there exists a positive definite matrix $P > 0$
such that
\begin{subequations}\label{quadratic:eq1}
\begin{align}\label{quadratic:eq1all}
\forall q \in Q: \mathbf{S}(q,\Sigma,P)< 0,
\end{align}
where
\begin{itemize}
\item in continuous-time (Lyapunov equation)
\begin{equation}
\label{quadratic:eq1c}
\mathbf{S}(q,\Sigma,P)=A_{q}^T P+ P A_{q},
\end{equation} 
\item in discrete-time (Stein equation)
\begin{equation}
\label{quadratic:eq1d}
\mathbf{S}(q,\Sigma,P)=A_{q}^T PA_{q}-P. 
\end{equation} 
\end{itemize}
\end{subequations}
\end{Definition}
It is well-known \cite{D:Lib} that quadratic stability implies
exponential stability for all switching signals.  For our purposes
quadratic stability is convenient, as it implies the existence of an
$\mathcal{L}_2$ gain (Definition~\ref{def:l2}) and
controllability/observability grammians.

\begin{Definition}(\emph{Controllability/observability grammians}) \\
An \emph{observability grammian} of $\Sigma$ is a positive definite
solution $\Q > 0$ of the following inequality
\begin{subequations}\label{ContObsGram}
\begin{align}\label{ContObsGramall}
\forall q \in Q: \mathbf{O}(q,\Sigma,\Q) \le 0,
\end{align}
where
\begin{itemize}
\item in continuous-time
\begin{equation}\label{ContObsGramc}
\mathbf{O}(q,\Sigma,\Q)=A^T_q\Q+\Q A_q +C_q^TC_q,  
\end{equation} 
\item in discrete-time
\begin{equation}\label{ContObsGramd}
\mathbf{O}(q,\Sigma,\Q)=A^T_q\Q A_q +C_q^TC_q -\Q.  
\end{equation} 
\end{itemize}
\end{subequations}

A \emph{controllability grammian} of $\Sigma$ is a positive definite
solution $\CP > 0$ of the following inequality
\begin{subequations}\label{ContContrGram}
\begin{align}\label{ContContrGramall}
\forall q \in Q: \mathbf{C}(q,\Sigma,\CP) \le  0,
\end{align}
where
\begin{itemize}
\item in continuous-time
\begin{equation}\label{ContContrGramc}
\mathbf{C}(q,\Sigma,\CP)=A_q\CP+\CP A_q^T +B_qB^T_q, 
\end{equation} 
\item in discrete-time
\begin{equation}\label{ContContrGramd}
\mathbf{C}(q,\Sigma,\CP)=A_q\CP A_q^T +B_qB^T_q -\CP.
\end{equation} 
\end{itemize}
\end{subequations}
We will call the eigenvalues of the product $\CP\Q$ the \emph{singular
  values} of the pair of grammians $(\CP,\Q)$.
\end{Definition}
Existence of a controllability or observability grammian does not
imply quadratic stability.
 However, if in \eqref{ContContrGramall} or \eqref{ContObsGramall} we replace
 inequality by strict inequality, then existence of a positive definite solution to
 the thus obtaines LMIs is equivalent to quadratic stability.
More precisely, using techniques from \cite{BoydBook},
one can show that
\begin{Lemma}
\label{BalancClaim9}
 The following are equivalent:
 \begin{itemize}
 \item{\textbf{(i)}}
 $\Sigma$ is quadratically stable, 
 \item{\textbf{(ii)}}
 there exists $\Q > 0$ such that $\forall q \in Q: \mathbf{O}(q,\Sigma,\Q) < 0$,
 \item{\textbf{(iii)}}
 there exists $\CP > 0$ such that $\forall q \in Q: \mathbf{C}(q,\Sigma,\CP) < 0$,
 \end{itemize}
\end{Lemma}
The proof of Lemma \ref{BalancClaim9} can be found in
\ref{appl1}.  

Note that existence of controllability grammian does not
imply controllability, even if $|Q|=1$, i.e. we have the classical
linear case.  For a counter-example for the linear case, see
\cite{BeckThesis}.

To define the $\mathcal{L}_2$ norm for an \BLSS, we recall that
$||\cdot||_2$ denotes either the $l_2$ norm (in discrete-time) or the
$L_2$ norm (in continuous-time).

\begin{Definition}[\cite{HespanhaNov03}]\label{def:l2}
We say that $Y^{\Sigma}$ has an $\mathcal{L}_2$ gain $\gamma > 0$ if
\begin{enumerate}[a)]
\item $Y^{\Sigma}(u,q)$  belongs to $L_2(T,Y)$ (in continuous time case) or
      to $l_2(Y)$ (in discrete time case) for all $(u,q) \in\mathcal{U}\times\mathcal{Q}$,
and
\item
\begin{equation}\label{l2gain}
\sup_{q \in\mathcal{Q}} ||Y^{\Sigma}(u,q)||_{2} \le \gamma
||u||_{2}\quad\forall u\in\mathcal{U}. 
\end{equation}
\end{enumerate}
If $Y^{\Sigma}$ has an $\mathcal{L}_2$ gain, then we define the
$\mathcal{L}_2$ norm of $Y^{\Sigma}$, denoted by 
$||Y^{\Sigma}||_{\mathcal{L}_2}$, as the infimum of all $\gamma > 0$
such that \eqref{l2gain} holds.  If $Y^{\Sigma}$ does not have
an $\mathcal L_2$ gain, then we set
$||Y^{\Sigma}||_{\mathcal{L}_2}=+\infty$. 
\end{Definition}
Note that existence of a $\mathcal{L}_2$ gain of $Y^{\Sigma}$ means that
the outputs of $\Sigma$ belong to $L_2(T,Y)$ (cont. time) or
$l_2(Y)$ (disc. time) respectively.
We note that since $u\mapsto Y^{\Sigma}(u,q)$ is linear for each $q$,
the $\mathcal{L}_2$ norm of $Y^{\Sigma}$ can equivalently be defined
as
\[
||Y^{\Sigma}||_{\mathcal{L}_2}=\sup_{q
  \in\mathcal{Q}}~\sup_{||u||_2=1} ||Y^{\Sigma}(u,q)||_{2},
\]
whenever $||Y^{\Sigma}(u,q)||_{2}$ exists. In other words,
$||Y^{\Sigma}||_{\mathcal{L}_2}$ is the supremum of the operator norms
\[
||Y^{\Sigma}(\cdot,q)||=\sup_{||u||_2=1} ||Y^{\Sigma}(u,q)||_{2}.
\]

\begin{Lemma}
\label{BalancClaim6}
If $\Sigma$ is quadratically stable, then \eqref{BalancClaim6:eq1} below has a positive definite solution $P$.   
\begin{subequations}\label{BalancClaim6:eq1}
\begin{align}\label{BalancClaim6:eq1all}
\forall q \in Q: \mathbf{G}_\gamma(q,\Sigma,P)< 0,
\end{align}
where
\begin{itemize}
\item in continuous time
\begin{equation}\label{BalancClaim6:eq1c}
\mathbf{G}_\gamma(q,\Sigma,P)=
\begin{bmatrix}
A_q^T P + P A_q+C_q^TC_q & P B_q\\
B_q^T P & -\gamma^{2}I
\end{bmatrix},
\end{equation}

\item in discrete time
\begin{equation}\label{BalancClaim6:eq1d}
\mathbf{G}_\gamma(q,\Sigma,P)=\begin{bmatrix} A_q^TPA_q+C_q^TC_q-P, & A_q^TPB_q \\ 
B_q^TPA_q, & B^T_qPB_{q} -\gamma^2I 
\end{bmatrix}.
\end{equation}
\end{itemize}
\end{subequations}
If a
positive definite solution to \eqref{BalancClaim6:eq1} exists, then
for any $(u,q) \in \mathcal{U} \times \mathcal{Q}$, $||Y^{\Sigma}(u,q)||_{2}$ is defined and
$||Y^{\Sigma}||_{\mathcal{L}_2} \le \gamma$.
\end{Lemma}
The proof of Lemma \ref{BalancClaim6} can be found in
\ref{appl2}.

\section{Invariance of state-space representation}
\label{sect:inv}
In the previous section, we defined quadratic stability,
$\mathcal{L}_2$ gains, and grammians. The concepts were defined in
terms of LMIs. We will show that the existence of a solution to those
LMIs is a property of the input-output map. Furthermore, for
equivalent minimal systems, the set of solutions are isomorphic. In
order to formalize this result, we will introduce the following
notation.
\begin{Definition}
\label{def:sets}
For a \BLSS
\[\Sigma = \SwitchSysLin[{}],\]
define the following subsets of the set of $n \times n$ strictly
positive definite matrices
\begin{itemize}
\item $\mathbf{S}(\Sigma)$ is the ``stability" set of all $P > 0$ which satisfy
\eqref{quadratic:eq1}.
\item $\mathbf{O}(\Sigma)$ is the ``observability" set of all $\Q > 0$ for which
\eqref{ContObsGram} holds.
\item $\mathbf{C}(\Sigma)$ is the ``controllability" set of all $\CP > 0$ for which
\eqref{ContContrGram} holds.
\item For $\gamma > 0$, let $\mathbf{G}_{\gamma}(\Sigma)$ be the set
of all $P > 0$ which satisfy \eqref{BalancClaim6:eq1}.
\end{itemize}
\end{Definition}
 
Now, we can state the following result.
\begin{Theorem}
\label{inv:theo}
Let $\mathbf{K}$ be any symbol from
$\{\mathbf{S},\mathbf{O},\mathbf{C},\mathbf{G}_{\gamma}\}$.
\begin{enumerate}
\item If the \BLSS $\Sigma$ is such that $\mathbf{K}(\Sigma) \ne
\emptyset$, then for any minimal \BLSS $\Sigma_{\mathrm m}$ which is
equivalent to $\Sigma$, $\mathbf{K}(\Sigma_{\mathrm m}) \ne
\emptyset$.
\item
 If $\CP \in \mathbf{C}(\Sigma)$, $\Q \in \mathbf{O}(\Sigma)$, then
      for any minimal \LSS\ $\Sigma_m$ which is equivalent to $\Sigma$,
      there exist $\CP_m \in \mathbf{C}(\Sigma_m)$ and $\Q_m \in \mathbf{O}(\Sigma_m)$,
      such that the following holds: if $\sigma_1 \ge, \ldots, \ge \sigma_n$ are the
      singular values of $(\CP,\Q)$ and $\lambda_1 \ge, \ldots, \ge \lambda_k$ are
      the singular values of $(\CP_m,\Q_m)$, then $\sigma_{n-k+i} \le \lambda_i \le \sigma_i$,
      $i=1,\ldots,k$, $k=\dim \Sigma_m$.

\item Let $\Sigma_1$ and $\Sigma_2$ be two \SLSS ~of dimension $n$ and
let $\MORPH:\Sigma_1 \rightarrow \Sigma_2$ be an isomorphism between
them.  If $\mathbf{K} \in
\{\mathbf{S},\mathbf{O},\mathbf{G}_{\gamma}\}$ then define
${M}=\MORPH^{-1} \in \mathbb{R}^{n \times n}$; if
$\mathbf{K}=\mathbf{C}$, then define ${M}=\MORPH^T$.  Then
\begin{equation}
\label{inv:theo:eq1} 
P \in \mathbf{K}(\Sigma_1) \iff {M}^{T}P {M} \in \mathbf{K}(\Sigma_2).
\end{equation}
In particular, for any two minimal and equivalent \SLSS~ $\Sigma_1$
and $\Sigma_2$, there exists a non singular matrix ${M}$ such that
\eqref{inv:theo:eq1} holds.
\end{enumerate}
\end{Theorem}
The theorem above expresses that quadratic stability and existence of
controllability/observability grammians are preserved by minimality.
In fact, if one of these properties holds for a state-space
representation, then it holds for any minimal state-space representation.
From Theorem \ref{inv:theo} it follows that it is sufficient to perform balanced truncation on minimal systems. This will be explained in detail in Remark \ref{min:suff}.
\begin{Corollary}
For minimal \SLSS, the singular values of grammians do not depend on
the choice of state-space representation.  Indeed, assume that
$\Sigma_1$ is a minimal \LSS, and consider the singular values
$\sigma_1,\ldots,\sigma_n$ for a choice of grammians $(\CP_1,\Q_1)$ of
$\Sigma_1$.  Then for any minimal \LSS\ $\Sigma_2$ which is equivalent
to $\Sigma_1$ there exists a pair of grammians $(\CP_2,\Q_2)$ of
$\Sigma_2$ such that the singular values of $(\CP_2,\Q_2)$ are also
$\sigma_1,\ldots,\sigma_n$.
\end{Corollary}
 
For an \LSS\ $\Sigma$, define $\gamma(\Sigma)=\inf\{\gamma > 0 \mid
\mathbf{G}_{\gamma}(\Sigma) \ne \emptyset\}$. Then clearly the
$\mathcal{L}_2$ norm of the input-output map of $\Sigma$ is at most
$\gamma(\Sigma)$.  From Theorem \ref{inv:theo}, we obtain that

\begin{Corollary}~\\
\vspace{-.5cm}
\begin{itemize}
\item For any minimal \LSS\ $\Sigma_m$ which is equivalent to
$\Sigma$, $\gamma(\Sigma_m) \le \gamma(\Sigma)$.
\item If $\Sigma_i$, $i=1,2$, are two minimal and equivalent \SLSS,
then $\gamma(\Sigma_1)=\gamma(\Sigma_2)$.
\end{itemize}
\end{Corollary}
As a consequence, the number $\gamma(\Sigma)$, where $\Sigma$ is
minimal, depends only on the input-output map of $Y^{\Sigma}$.  Note
that $\gamma(\Sigma)$ can be computed by solving a classical
optimization problem.

\begin{proof}[Proof of Theorem \ref{inv:theo}]
The proof for both the discrete- and the continuous-time case are the
same, hence we present both cases together.

The proof of the last part of the theorem follows by an easy
computation and by recalling that if $\Sigma_1$ and $\Sigma_2$ are two
equivalent and minimal \SLSS, then they are related by an \LSS\
isomorphism.

Next, we prove the first statement of the theorem.
In order to make the proof easier, we denote by $\mathbf{\hat{C}}(\Sigma)$ the set of all inverses of elements of $\mathbf{C}(\Sigma)$.
Clearly, showing that $\mathbf{\hat{C}}(\Sigma) \ne \emptyset \implies \mathbf{\hat{C}}(\Sigma_{\mathrm m}) \ne \emptyset$ is equivalent to showing $\mathbf{\hat{C}}(\Sigma) \ne \emptyset \implies \mathbf{\hat{C}}(\Sigma_{\mathrm m}) \ne \emptyset$.
Hence, in the sequel, we will show that
$\mathbf{K}(\Sigma) \ne \emptyset \implies \mathbf{K}(\Sigma_{\mathrm m}) \ne \emptyset$ for
$\mathbf{K} \in \{\mathbf{S},\mathbf{\hat{C}},\mathbf{O},\mathbf{G}_{\gamma}\}$.
To this end, it is enough to show that if $\mathbf{K}(\Sigma) \ne \emptyset$
and we apply Procedures \ref{LSSreach}--\ref{LSSobs} to obtain a
minimal \LSS\ $\Sigma_{\mathrm m}$, then $\mathbf{K}(\Sigma_{\mathrm m}) \ne \emptyset$. 



First, we show that the application of Procedure \ref{LSSreach}
preserves the non-emptiness of $\mathbf{K}(\Sigma)$.  Recall the
partitioning of $A_q$ from \eqref{LSSreach:eq1} and consider the
corresponding partitioning of $P$
\[ P=\begin{bmatrix} P_{11} & P_{12} \\ P_{21} & P_{22} \end{bmatrix}.
\]

A simple computation reveals that
\begin{Lemma}
\label{inv:theo:lemma2}
If $P \in \mathbf{K}(\Sigma)$, then $P_{11} \in \mathbf{K}(\Sigma_r)$
for $\mathbf{K} \in \{\mathbf{S},\mathbf{G}_{\gamma},\mathbf{O},\mathbf{\hat{C}}\}$.
\end{Lemma}
The proof of Lemma \ref{inv:theo:lemma2} can be found in \ref{appl3}.

Next, we show that Procedure \ref{LSSobs} preserves non-emptiness of
$\mathbf{K}(\Sigma)$.  We will use the duality between observability
and reachability, explained in Remark \ref{rem:dual}. Define the dual system
$\Sigma^T=(n,Q,\{(A_q^T,C_q^T,B_q^T) \mid q \in Q\})$.  The following
properties of the dual system $\Sigma^T$ hold.
\begin{Lemma}
\label{inv:theo:lemma3}~\\
\vspace{-.5cm}
\begin{enumerate}
\item For $P \in \mathbf{S}(\Sigma) \iff P^{-1} \in
\mathbf{S}(\Sigma^T)$,
\item
 $P \in \mathbf{O}(\Sigma) \iff P^{-1} \in \mathbf{\hat{C}}(\Sigma^T)$ and
 $P \in \mathbf{\hat{C}}(\Sigma) \iff P^{-1} \in \mathbf{O}(\Sigma^T)$.
\item
$P \in \mathbf{G}_{\gamma}(\Sigma) \iff \gamma^{2}P^{-1} \in
\mathbf{G}_{\gamma}(\Sigma^T).$
\end{enumerate}
\end{Lemma}
The proof of Lemma \ref{inv:theo:lemma3} can be found in
\ref{appl4}.  

As a consequence, $\mathbf{K}(\Sigma) \ne \emptyset$ if and only if
$\mathbf{K}(\Sigma^T) \ne \emptyset$ for $\mathbf{K} \in \{\mathbf{S},\mathbf{G}_{\gamma}\}$,
and $\mathbf{O}(\Sigma) \ne \emptyset \implies \mathbf{\hat{C}}(\Sigma^T) \ne \emptyset$,
and $\mathbf{\hat{C}}(\Sigma) \ne \emptyset \implies \mathbf{O}(\Sigma^T) \ne \emptyset$.
  From the definition of duality
it follows that if $\Sigma_{\mathrm rt}$ is the result of applying
Procedure \ref{LSSreach} to $\Sigma^T$, then $\Sigma_{\mathrm{rt}}^T=
\Sigma_{\mathrm o}$, where $\Sigma_{\mathrm o}$ is the result of
application of Procedure \ref{LSSobs} to $\Sigma$.  Since Procedure
\ref{LSSreach} preserves non-emptiness of $\mathbf{K}(\Sigma^T)$,
$\mathbf{K} \in \{\mathbf{S},\mathbf{G}_{\gamma},\mathbf{O},\mathbf{\hat{C}}\}$, we
have that
$\mathbf{K}(\Sigma) \ne \emptyset \implies \mathbf{K}(\Sigma^T) \ne \emptyset \implies \mathbf{K}(\Sigma_{\mathrm{rt}})  \ne \emptyset \implies \mathbf{K}(\Sigma_{\mathrm o}) \ne \emptyset$,
$\mathbf{K} \in \{\mathbf{S},\mathbf{G}_{\gamma}\}$, and
$\mathbf{\hat{C}}(\Sigma) \ne \emptyset \implies \mathbf{O}(\Sigma^T) \ne \emptyset \implies \mathbf{O}(\Sigma_{\mathrm{rt}}) \ne \emptyset \implies \mathbf{\hat{C}}(\Sigma_{\mathrm o}) \ne \emptyset$,  and
$\mathbf{O}(\Sigma) \ne \emptyset \implies \mathbf{\hat{C}}(\Sigma^T) \ne \emptyset \implies \mathbf{\hat{C}}(\Sigma_{\mathrm{rt}}) \ne \emptyset \implies \mathbf{O}(\Sigma_{\mathrm o}) \ne \emptyset$.

Finally, we show the second statement. Without loss of generality, we can 
assume that $\Sigma_m$ is the result of applying Procedure \ref{LSSreach} and
Procedure \ref{LSSobs} to $\Sigma$. Consider 
a matrix pair $(\CP^{'},\Q^{'})$, $(\CP,\Q)$,
$\CP^{'},\Q^{'} \in \mathbb{R}^{k \times k}$
$\CP,\Q \in \mathbb{R}^{n \times n}$, $\CP,\CP^{'},\Q,\Q^{'} > 0$.
Let $\lambda_1 \ge \ldots \ge \lambda_n$ be the eigenvalues of $\CP\Q$ and
let $\lambda^{'}_1 \ge \ldots \ge \lambda^{'}_k$ be the eigenvalues of
$\CP^{'}\Q^{'}$.
We write $(\CP^{'},\Q^{'}) \preceq (\CP,\Q)$, if $k \le n$ and 
$\lambda_{n-k+i} \le \lambda^{'}_i \le \lambda_i$,
$i=1,\ldots,k$.
It is easy to see that $\preceq$ is a  transitive relation.
Hence, it is enough to show that if
we apply Procedure \ref{LSSreach} or Procedure \ref{LSSobs}
to $\Sigma$, then the resulting system will have a pair of controllability and
observability grammians $(\CP^{'},\Q^{'})$, such that 
$(\CP^{'},\Q^{'}) \preceq (\CP,\Q)$. 
From Lemma \ref{inv:theo:lemma3} it follows that it is enough to prove this only
for Procedure \ref{LSSreach}. 

Let $P_{11}$ and $\Q_{11}$ be the
upper left $r \times r$ sub-matrices of $P=\CP^{-1}$ and $\Q$ in the basis described
in Procedure \ref{LSSreach}. Then Lemma \ref{inv:theo:lemma2} implies
that $P^{-1}_{11} \in \mathbf{C}(\Sigma_r)$ and 
$\Q_{11} \in \mathbf{O}(\Sigma_r)$.   
We claim that $(P^{-1}_{11},\Q_{11}) \preceq (\CP,\Q)$. 
Indeed, notice that the eigenvalues of $\CP\Q$ and $P^{-1}_{11}\Q_{11}$ are
exactly the characteristic values of the regular matrix pencils
$(\Q-\lambda \CP^{-1})$ and  
$(\Q_{11}-\lambda P_{11})$ respectively, see \cite[Chapter X.\S6]{GantmacherBook}.
 But  $(\Q_{11}-\lambda P_{11})$ is just the pencil $(\Q-\lambda \CP^{-1})$ with
 $n-r$ independent linear 
 constraints which describe the orthogonal complement of $\mathcal{V}^{*}(\Sigma)$.
 Then from \cite[Chapter X.\S7, Theorem 14]{GantmacherBook}, 
it follows that $(P_{11}^{-1},\Q_{11}) \preceq (\CP,\Q)$.
\end{proof}

\section{System-theoretic interpretation of grammians and their
  singular values}
\label{sys:interp}
In this section, we provide a system theoretic interpretation of
grammians and their singular value. To this end, we link them to observability,
span-reachability and Hankel-norms.

\begin{Theorem}
\label{BalancClaim4}
If $\CP$ be a positive semi-definite matrix 
which satisfies \eqref{ContObsGramall} and
$\Sigma$ is span-reachable, then $\CP$ is positive definite.
Similarly, if $\Q$ is a positive semi-definite matrix
which satisfies \eqref{ContContrGramall} 
and $\Sigma$ is observable, then $\Q$ is positive definite. 
\end{Theorem}
The proof of Theorem \ref{BalancClaim4} is based on the following
results, which are interesting on their own right.

\begin{Lemma}
\label{BalancClaim1}
Assume that $\Q \ge 0$ satisfies $\forall q \in Q: \mathbf{O}(q,\Sigma,\Q) \le 0$.  
Then for
all $q \in \mathcal{Q}$, $t > 0$
\begin{align*}
x^T\Q x \ge
\begin{cases}
\int_{0}^{t} ||Y_{x}^{\Sigma}(0,q)(s)||^2_{2}ds&\text{(cont.)}\\
\sum_{s=0}^{t} ||Y_{x}^{\Sigma}(0,q)(s)||^2_{2}&\text{(disc.)}
\end{cases}
\end{align*}
\end{Lemma}

\begin{Lemma}
\label{BalancClaim3}
Assume that $\CP > 0$ is a solution to $\forall q \in Q: \mathbf{C}(q,\Sigma,\CP) \le 0$.  Then
for all $(u,q)\in\mathcal{U}\times\mathcal{Q}$, and $t > 0$
\[ x^T\CP^{-1}x \le
\begin{cases}
\int_{0}^{t} ||u(s)||_{2}^{2}ds\\
\sum_{k=0}^{t} ||u_k||_{2}^{2}~,
\end{cases}
\]
where $x=X^{\Sigma}_{0}(u,q)(t)$.
\end{Lemma}
The proofs of Lemma \ref{BalancClaim1} and \ref{BalancClaim3} can be
found in \ref{appl5} and \ref{appl6}.

\begin{proof}[Proof of Theorem \ref{BalancClaim4}]
We prove the statement for the observability by contradiction. Assume
that there exists $x \in \mathbb{R}^n \setminus \{0\}$ such that
$x^T\Q x = 0$. By Lemma \ref{BalancClaim1}, this implies that for all
$q \in \mathcal{Q}$, the map $Y_x^{\Sigma}(0,q)=0$ and thus
$Y_x^{\Sigma}(0,q)=Y_{0}^{\Sigma}(0,q)$ for all $q$. Note that
$Y_x^{\Sigma}(u,q)=Y_{x}^{\Sigma}(0,q)+Y_{0}^{\Sigma}(u,q)$, and hence
we get that $Y_{x}^{\Sigma}(u,q)=Y_0^{\Sigma}(u,q)$ for all $q$ and
$u$, which contradicts the observability of $\Sigma$.

The statement for controllability grammian follows by duality.
\end{proof}
 
Notice that an \LSS\ may fail to be
observable (resp. reachable), even if \eqref{ContObsGram}
(resp. \eqref{ContContrGram}) has a positive definite
solution, see Example \ref{example1} of Section \ref{sect:model_red}.
This is due to the fact that in \eqref{ContObsGram} and
\eqref{ContContrGram}, we require inequalities instead of equalities.

As we shall see in Section \ref{sect:model_red}, a side effect of
this phenomenon is that the reduced order model obtained by balanced
truncation may fail to be minimal.

As a consequence, one is tempted to ask the question what happens if
in \eqref{ContObsGram} (resp. in \eqref{ContContrGram}), we require
that for some or for all $q \in Q$, $A_{q}^T\Q+\Q A_q+C_q^TC_q=0$
(resp. $\CP A^T_q+A_q\CP+B_qB_q^T=0)$ holds with equality. In this
case, the existence of a strictly positive definite solution to the
equations will imply observability (resp. controllability) of some (or
all) linear subsystems. However, in Remark \ref{lin:rem2}, we have
already explained that for a large class of input-output maps,
including those which are realizable by quadratically stable \SLSS,
there exist no state-space representation such that the local linear
subsystems are reachable or observable.
Hence, by replacing inequalities by
equalities we necessarily restrict applicability of the model
reduction approach.

\subsection{Nice grammians}

As it was mentioned before, existence of a strictly positive definite
controllability and observability grammian does not imply span-reachability
and observability. In fact, there might exist many grammians, and it
is not clear which one of them should be chosen.  In discrete
time, the problem above can partially be circumvented, by using what
we will call \emph{nice grammians}. Nice grammians are special cases
of grammians which have the property that they are unique and their
existence is equivalent to observability and controllability of the
system. Their disadvantage is that they are not preserved by balanced
truncation. However, they are potentially interesting for computational
purposes and as canonical forms in system identification, and for this
reason we discuss them here.

For the formal definition, we introduce the following terminology.
\begin{Definition}[Strong stability]
We call the \LSS\ $\Sigma$ \emph{strongly} \emph{stable}, if the matrix
$\sum_{q \in Q} A_q^T \otimes A_q^T$ is a stable matrix (all its
eigenvalues lie inside the unit disc).
\end{Definition}
From \cite{CostaBook}, we obtain the following result.
\begin{Lemma}
\label{nice_stab:proof}
Consider $n \times n$ matrices $F_q$, $q \in Q$.  If $\sum_{q \in Q}
F_{q}^T \otimes F^T_q$ is a stable matrix, then the equation
\[ P=(\sum_{q \in Q} F^T_qPF_q)+ \mathcal{G} \] has a positive semi-definite solution $P
\ge 0$ for all positive semi-definite $\mathcal{G} \ge 0$.  In fact,
this solution is unique and
\[
P=\sum_{w \in Q^{*}} F^T_{w}\mathcal{G}F_{w}.
\]
Conversely, if the inequality
\[ P-(\sum_{q \in Q} F^T_qPF_q) > 0 \] has a positive definite
solution, then $\sum_{q \in Q} F_{q}^T \otimes F_q^T$ is a stable
matrix.
\end{Lemma}
The proof of Lemma \ref{nice_stab:proof} can be found in 
\ref{appl7}.  Just like quadratic stability, strong stability is preserved
by minimization.
\begin{Lemma}
\label{nice_stab:min}
If $\Sigma$ is strongly stable and $\Sigma_m$ is a minimal \BLSS which
is equivalent to $\Sigma$, then $\Sigma_m$ is strongly stable too.
\end{Lemma}
The proof of Lemma \ref{nice_stab:min} can be found in 
\ref{appl8}.  With the above discussion in mind, we define the concept
of \emph{nice grammians}.
\begin{Definition}
\label{DTniceGram}
Assume that $\Sigma$ is strongly stable. Then the unique positive
semi-definite solutions $\CP$ and $\Q$ to
\begin{eqnarray*}
\CP&=&\sum_{q \in Q} A_q\CP A_q^T+\sum_{q \in Q} B_qB_q^T 
\label{niceContrGram} \\
\Q&=&\sum_{q \in Q} A_q^T\Q A_q+\sum_{q \in Q} C_q^TC_q 
\label{niceObsGram} 
\end{eqnarray*}
are called \emph{nice controllability} and \emph{nice observability}
grammians, respectively.
\end{Definition}
Notice that $A_q\CP A_q^T+B_qB_q^T \le \sum_{\sigma \in Q}
A_{\sigma}\CP A_{\sigma}^T+B_{\sigma}B_{\sigma}^T$ and $A_q^T\Q
A_q^T+C^T_qC_q \le \sum_{\sigma \in Q} A^T_{\sigma}\Q
A_{\sigma}+C^T_{\sigma}C_{\sigma}$.  Hence, nice grammians are indeed
grammians, since they satisfy \eqref{ContContrGram} and
\eqref{ContObsGram} respectively.
\begin{Lemma}
\label{niceGramLemma0}
If $\Sigma$ is strongly stable, then the nice controllability and
observability grammians exists and they are unique.
\end{Lemma}
The proof of Lemma \ref{niceGramLemma0} can be found in
\ref{appl9}.
\begin{Lemma}
\label{niceGramLemma1}
$\Sigma$ is span-reachable if and only if the nice controllability grammian
$\CP$ is strictly positive definite.  $\Sigma$ is observable if and
only if the nice observability grammian is strictly positive definite.
\end{Lemma}
\begin{proof}[Proof of Lemma \ref{niceGramLemma1}]
We prove the statement about observability, its counterpart on
span-reachability follows by duality.  From Lemma \ref{nice_stab:proof}, it
follows that
\[ x^T\Q x=\sum_{v \in Q^{*}}
x^TA_{v}^T\widetilde{C}^T\widetilde{C}A_{v}x= \sum_{v \in Q^{*}}
||\widetilde{C}A_vx||_{2}^{2}
\]
where $\widetilde{C}=\begin{bmatrix} C_1^T, & \ldots, &
C_{\QNUM}^T \end{bmatrix}^T$. Then it follows that $x^T\Q x=0$ is
equivalent to $C_qA_vx=0$ for all $q \in Q$, $v \in Q^{*}$.  By Remark \ref{lin:rem1.5} the
latter is equivalent to $x \in \ker\mathcal{O}(\Sigma)$. From this the
statement of the lemma follows by using Theorem
\ref{sect:problem_form:reachobs:prop1} and Remark \ref{lin:rem1.5}.
\end{proof}
To sum up, if $\Sigma$ is a strongly stable system, then the
minimization procedure Procedure \ref{LSSmin} preserves strong
stability. Moreover, if $\Sigma$ and $\hat{\Sigma}$ are two isomorphic
systems related by an isomorphism $\MORPH$ from $\Sigma$ to
$\hat{\Sigma}$, and if $\CP$ and $\Q$ are the nice grammians of
$\Sigma$, then $\MORPH \CP \MORPH^T$ and $\MORPH^{-T}\Q\MORPH^{-1}$ are the
nice controllability and observability grammians of $\hat{\Sigma}$.
In other words, the singular values of $\CP\Q$ are independent of the
choice of the particular minimal realization, and existence of
strictly positive definite nice observability and controllability
grammians is guaranteed in minimal systems.

\subsection{Singular values}
\label{SS:SV}
We present the interpretation of the largest singular value
of a grammian pair $(\CP,\Q)$ in terms of the Hankel-norm of the
input-output map.
\begin{Definition}[Hankel-norm]
Let $\Sigma$ be a quadratically stable \BLSS and define the \emph{Hankel-norm}
$||Y^{\Sigma}||_{H}$ as follows. Denote by $\mathcal{U}^{0}$ the set of all
inputs $u \in \mathcal{U}$ such that there exists a time instant $t$
($t \in T$ in continuous-time, $t \in \mathbb{N}$ for discrete time), such that
$u(s)=0$ for all $s > t$. 
 Let $\mathbf{HG}(\Sigma)$ be the set
of all $\gamma > 0$ such that
\[ \forall u \in \mathcal{U}^0, q \in \mathcal{Q}: 
||Y^{\Sigma}(u,q)||_{2} \le \gamma ||u||_{2},
\]
Define the Hankel-norm $||Y^{\Sigma}||_{H}$ of $Y^{\Sigma}$ as 
$||Y^{\Sigma}||_{H}=\inf \mathbf{HG}(\Sigma)$.
\end{Definition}
\begin{Remark}
 If $\Sigma$ is quadratically stable, then by Lemma \ref{BalancClaim6},
 $Y^{\Sigma}$ has a finite $\mathcal{L}_2$ norm and the 
 $L_2$ (or in discrete-time case $l_2$) norm $||Y^{\Sigma}(u,q)||_{2}$ exists and
 it is finite.  From the definition of $||Y^{\Sigma}||_{\mathcal{L}_2}$ it follows
 that $||Y^{\Sigma}||_{\mathcal{L}_2} \in \mathbf{HG}(\Sigma)$ and 
 hence $||Y^{\Sigma}||_H \le ||Y^{\Sigma}||_{\mathcal{L}_2}$.
\end{Remark}

Intuitively, the Hankel-norm of $Y^{\Sigma}$ corresponds to the
maximum output energy of the system, if we first feed in a continuous
input $u$ with unit energy and from some time $t$ we stop feeding in
continuous input and we let the system to develop autonomously.
\begin{Theorem}
\label{BalancClaim7}
Consider an \BLSS $\Sigma$. Assume that $\CP > 0$ is a controllability
grammian and $\Q> 0$ is an observability grammian of $\Sigma$.  The
largest singular value $\sigma_{max}$ of $(\CP,\Q)$ satisfies
\[ ||Y^{\Sigma}||_{H} \le \sigma_{max}. \]
\end{Theorem}
\begin{proof}[Proof of Theorem \ref{BalancClaim7}]
Pick a switching signal and input $(q,u) \in \mathcal{Q} \times \mathcal{U}$ 
such that for some time instance $t$
that $u(s)=0$ for all $s > t$ and $||u||_{2} \le 1$.  Denote by $x$
and $y$ the corresponding state and output trajectories.
Note that by Lemma \ref{BalancClaim6},
$y$ belongs to $L_2(T,Y)$ (cont. time) or $l_2(Y)$ (disc. time) and 
hence the norm $||y||_2$ exists and it is finite.
 By combining
Lemma \ref{BalancClaim1} and Lemma \ref{BalancClaim3}, we obtain that
$x(t)^T\CP^{-1}x(t) \le 1$
and
\[ x^T(t)\Q x(t) \ge ||y||_{2}^{2}
\]
Since $u$, $q$ and $t$ are arbitrary, we then obtain that
\[ \sup_{x^T\CP^{-1}x \le 1} x^T\Q x \ge ||Y^{\Sigma}||^2_{H} \] We
proceed to prove that
\[ \lambda_{max}(\CP\Q) = \sup_{x^T\CP^{-1}x \le 1}x^T\Q x \]
Let $\CP^{-1}=S^TS$, and define $\hat{\Q}=(S^{-1})^T\Q S^{-1}$.
It follows that
\[ \{Sx \mid x^T\CP^{-1}x \le 1\} = \{ v \mid v^Tv \le 1\}. \] Hence,
\[ \sup_{x^T\CP^{-1}x \le 1}x^T\Q x=\sup_{v^Tv \le 1} v^T\hat{\Q}v =
\lambda_{\max}(\hat{\Q}),
\]
where $\lambda_{max}(\hat{Q})$ is the maximal eigenvalue of $\hat{\Q}$. But
$\hat{\Q}=S\CP\Q S^{-1}$, hence the eigenvalues of $\hat{\Q}$ and
$\CP\Q$ coincide.
\end{proof}

For nice grammians $\CP$ and $\Q$, the sum of their singular values
can be represented as a Frobenius norm of the Hankel-matrix of the
system.  In order to present this result, we will first recall the
notion of a Hankel-matrix for linear switched systems. Although
Hankel-matrices can be defined both for continuous- and discrete-time,
we will below concentrate on the discrete-time case only.  Consider a
\LSS\ $\Sigma=\SwitchSysLin$ and for any $v \in Q^{*}$, define the
\emph{Markov-parameter} $M_{v}=\widetilde{C}A_{v}\widetilde{B}$, where
$\widetilde{C}=\begin{bmatrix} C_1^T, & \ldots, &
C_{\QNUM}^T \end{bmatrix}^T$ and $\widetilde{B}=\begin{bmatrix} B_1^T,
& \ldots, & B_{\QNUM} \end{bmatrix}$.  In \cite{MPLBJH:Real} it was
shown that the definition of the Markov-parameter depends only on the
input-output map $Y^{\Sigma}$ and not on the choice of $\Sigma$
itself. Define then the \emph{Hankel-matrix} of $Y^{\Sigma}$ as the
infinite block matrix
\[ H=(H_{s,v})_{s,v \in Q^{*}}, \] rows and columns of which are
indexed by sequences from $Q^{*}$ such that \( H_{s,v}=M_{vs} \).
Using the notation above, we can relate $\mathrm{tr}(\CP\Q)$ and the
Frobenius-norm of $H$ as follows.
\begin{Lemma}\label{lemmaproof}
If $\CP$ and $\Q$ are the nice grammians of $\Sigma$, then
\[ \mathrm{tr}(\CP\Q)=\sum_{v,s \in Q^{*}} ||H_{s,v}||^{2}_{F}, \] where
$||.||_F$ denotes the Frobenius norm of a matrix.
\end{Lemma}
\begin{proof}[Proof of Lemma~\ref{lemmaproof}]
Notice that $$\Q=\sum_{q \in Q} A^T_{q}\Q
A_{q}+\widetilde{C}^T\widetilde{C} \hbox{ and } \CP=\sum_{q \in Q} A_{q}\CP
A_q^T+\widetilde{B}\widetilde{B}^T.$$  From Lemma \ref{nice_stab:proof}
it follows that $$\Q=\sum_{w \in Q^{*}}
A^T_{w}\widetilde{C}^T\widetilde{C}A_{w}~\text{and}~\CP=\sum_{v \in Q^{*}}
A_{v}\widetilde{B}\widetilde{B}^TA_{v}^T$$ and hence,
\[ \CP \Q = \sum_{w \in Q^{*}}\sum_{v \in Q^{*}}
A_{v}\widetilde{B}\widetilde{B}^TA_{v}^TA^T_{w}\widetilde{C}^T\widetilde{C}A_{w}, \]
and thus
\[ \mathrm{tr}(\CP\Q)=\sum_{w \in Q^{*}}\sum_{v \in Q^{*}}
\mathrm{tr}(A_{v}\widetilde{B}\widetilde{B}^TA_{v}^TA^T_{w}\widetilde{C}^T\widetilde{C}A_{w}).
\]
Notice now that
\[
\begin{split}
\mathrm{tr}((A_{v}\widetilde{B})(\widetilde{B}^TA_{v}^T & A^T_{w}\widetilde{C}^T)(\widetilde{C}A_{w}))\\
& =\mathrm{tr}(\widetilde{B}^TA_{v}^TA^T_{w}\widetilde{C}^T)(\widetilde{C}A_{w}A_{v}\widetilde{B})\\
& = \mathrm{tr}(H_{w,v}^TH_{w,v})\\
&=||H_{w,v}||^{2}_{F}
\end{split}
\]
\end{proof}

If $\CP=\Q=diag(\sigma_1,\ldots,\sigma_n)$, then
$\mathrm{tr}(\CP\Q)=\sum_{i=1}^{n} \sigma_i^2$ and hence the lemma above
implies that this quantity does not depend on the state-space
representation.

\section{Model reduction for linear switched systems}
\label{sect:model_red}
In this section, we state the procedure for model reduction by
balanced truncation, and we prove a bound of the approximation error.
\begin{Procedure}{\textbf{Balanced truncation}}
\label{BalancedTruncate}
Consider a \BLSS
$\Sigma=\SwitchSysLin$.
\begin{enumerate}
\item Find a positive definite solution $\Q > 0$ to
\eqref{ContObsGram}.
\item Find a positive definite solution $\CP > 0$ to
\eqref{ContContrGram}.
\item Find $U$ such that $\CP=UU^T$ and find an orthogonal $K$ such that $U^{T}\Q
U=K\Lambda^2 K^{T}$, where $\Lambda$ is diagonal with the diagonal
elements taken in decreasing order.  Define the transformation
\[ \MORPH =\Lambda^{1/2}K^TU^{-1} \]
\item Replace $\Sigma$ with $$\Sigma_{\mathrm{bal}}=(n, Q,
(\bar{A}_q=\MORPH A_q\MORPH^{-1},\bar{B}_q=\MORPH
B_q,\bar{C}_q=C_q\MORPH^{-1})_{q \in Q}).$$
\item

The transformed system $\Sigma_{\mathrm{bal}}$ is balanced, i.e.,
$\forall q \in Q: \mathbf{O}(q,\Sigma_{\mathrm{bal}},\Lambda) < 0$
$\forall q \in Q: \mathbf{C}(q,\Sigma_{\mathrm{bal}},\Lambda) < 0$.
Indeed, it is enough to notice that
$\Lambda=(\MORPH^{-1})^T\Q\MORPH^{-1}=\MORPH \CP \MORPH^{T}$ and use these
expressions to derive 
$\mathbf{O}(q,\Sigma_{\mathrm{bal}},\Lambda) < 0$ and
$\mathbf{C}(q,\Sigma_{\mathrm{bal}},\Lambda) < 0$ from
$\mathbf{O}(q,\Sigma,\Q) < 0$ and
$\mathbf{C}(q,\Sigma,\CP) < 0$ for all $q \in Q$.

\item Assume that $\Lambda=\diag(\sigma_1,\ldots,\sigma_n)$, $\sigma_1
\ge \sigma_2 \ge \cdots \ge \sigma_n$.  Choose $r < n$ and let
$\Lambda_1=\diag(\sigma_1,\ldots,\sigma_r)$. Choose $\hat{A}_q \in
\mathbb{R}^{r \times r}$, $\hat{B}_q \in \mathbb{R}^{r \times m}$ and
$\hat{C}_q \in \mathbb{R}^{p \times r}$ so that
\begin{equation}
\label{BalancedTruncate:eq1}
\bar{A}_q=\begin{bmatrix}
\hat{A}_q & A_{q,12} \\
A_{q,21} & A_{q,22}
\end{bmatrix} \mbox{, \ \ } \bar{B}_q = \begin{bmatrix} \hat{B}_q \\
B_{q,2}
\end{bmatrix} \mbox{,\ \ \ } \bar{C}_q^T = \begin{bmatrix} \hat{C}_q^T
\\ C_{q,2}^T \end{bmatrix}
\end{equation}
Return as a reduced order model $\hat{\Sigma}=(r, Q, \{
(\hat{A}_q,\hat{B}_q,\hat{C}_q) \mid q \in Q\})$.
\end{enumerate}
\end{Procedure}

In the following, we will state an error bound for the difference
between the input-output maps of $\Sigma$ and $\hat{\Sigma}$. To this
end, we will use the following simple fact, due to \cite{Shaker1} in
the continuous case.
\begin{Lemma}
\label{BalancedTruncateLemma1}
The \BLSS $\hat{\Sigma}$ returned by Procedure \ref{BalancedTruncate}
is balanced.
In addition, if either $\forall q \in Q: \mathbf{O}(q,\Sigma,\Q) < 0$ or
$\forall q \in Q: \mathbf{C}(q,\Sigma,\CP) < 0$, then
$\hat{\Sigma}$ is also quadratically stable.
\end{Lemma}
\begin{proof}[Proof of Lemma \ref{BalancedTruncateLemma1}]
 The first statement of the proof follows by
 showing that 
$\forall q \in Q: \mathbf{O}(q,\hat{\Sigma},\Lambda_1) \le 0$ and
$\forall q \in Q: \mathbf{C}(q,\hat{\Sigma},\Lambda_1) \le 0$.
For the continous-time case the proof of this claim is straightforward:
\[
  \begin{split}
   & \bar{A}_q^T\Lambda = \begin{bmatrix} \hat{A}_q\Lambda_1 & \star \\
                                        \star & \star
                       \end{bmatrix} \\
   & \bar{A}_q\Lambda = \begin{bmatrix} \hat{A}_q\Lambda_1 & \star \\
                                        \star & \star
                       \end{bmatrix} \\
 \end{split}
\]
and hence for $\mathbf{K} \in \{\mathbf{C},\mathbf{O}\}$,
\[
  \mathbf{K}(q,\bar{\Sigma},\Lambda)=
  \begin{bmatrix} \mathbf{K}(q,\hat{\Sigma},\Lambda_1) & \star \\
                   \star & \star \end{bmatrix}.
\]
Hence, if $\mathbf{K}(q,\bar{\Sigma},\Lambda) \le 0$, then
$\mathbf{K}(q,\hat{\Sigma},\Lambda_1) \le 0$.

For the discrete-time case, 
using the notation of Procedure \ref{BalancedTruncate}, it is easy to
see that if $\Lambda=\begin{bmatrix} \Lambda_1 & 0 \\ 0 &
\Lambda_2 \end{bmatrix}$ where $\Lambda_1 \in \mathbb{R}^{r \times
  r}$, then
\[
\begin{split}
& \bar{A}_{q}\Lambda \bar{A}_q^T =
\begin{bmatrix}
\hat{A}_q\Lambda_1\hat{A}_q + A_{q,12}\Lambda_2A_{q,12}^T, &
\star \\
\star, & \star
\end{bmatrix}  \\
& \bar{A}^T_{q}\Lambda \bar{A}_q =
\begin{bmatrix}
\hat{A}^T_q\Lambda_1\hat{A}_q + A_{q,21}^T\Lambda_2A_{q,21}, &
\star \\
\star, & \star
\end{bmatrix}  \\
\end{split}
\]
Hence, the corresponding inequalities for the balanced system
$\bar{\Sigma}$ can be written as
\[
\begin{split}
0 &\ge \bar{A}_q\Lambda\bar{A}_q^T+\bar{B}_q\bar{B}^T_q - \Lambda\\
& = \begin{bmatrix}
\hat{A}_{q}\Lambda_1\hat{A}_q^T+ A_{q,12}\Lambda_2A_{q,12}^T  +\hat{B}_q\hat{B}_q^T-\Lambda_1, & \star \\
\star & \star \end{bmatrix} \\
0 &\ge \bar{A}_q^T\Lambda\bar{A}_q+\bar{C}_q\bar{C}^T_q - \Lambda\\
& = \begin{bmatrix}
\hat{A}^T_{q}\Lambda_1\hat{A}_q+ A_{q,21}^T\Lambda_2A_{q,21}  +\hat{C}^T_q\hat{C}_q^T-\Lambda_1, & \star \\
\star & \star \end{bmatrix} \\
\end{split}
\]
Thus, we get that
\[
\begin{split}
\hat{A}_{q}\Lambda_1\hat{A}_q^T&+\hat{B}_q\hat{B}_q^T-\Lambda_1   \\
& \le \hat{A}_{q}\Lambda_1\hat{A}_q^T+ A_{q,12}\Lambda_2A_{q,12}^T  +\hat{B}_q\hat{B}_q^T-\Lambda_1 \le 0 \\
\hat{A}_{q}^T\Lambda_1\hat{A}_q&+\hat{C}^T_q\hat{C}_q-\Lambda_1  \\
& \le \hat{A}_{q}^T\Lambda_1\hat{A}_q+ A_{q,21}^T\Lambda_2A_{q,21}
+\hat{C}^T_q\hat{C}_q-\Lambda_1 \le 0
\end{split}
\]
from which it follows that $\Lambda_1$ is both a controllability and
observability grammian.

The proof above also yields that if
  $\forall q \in Q: \mathbf{O}(q,\Sigma,\Q) < 0$ or $\forall q \in Q: \mathbf{C}(q,\Sigma,\CP) < 0$, then $\forall q \in Q: \mathbf{O}(q,\hat{\Sigma},\Lambda_1) < 0$ or respectively $\forall q \in Q: \mathbf{C}(q,\hat{\Sigma},\Lambda_1) < 0$.  By Lemma \ref{BalancClaim9} the latter implies that $\hat{\Sigma}$ is quadratically stable.
\end{proof}
\begin{Remark}[Nice observability/controllability grammian]
In the discrete-time case, one could take the nice observability and
controllability grammians from Definition \ref{DTniceGram} as inputs
for Procedure \ref{BalancedTruncate}. It is clear that the balancing
step then leads to a nice controllability and observability grammians
which are diagonal and equal to each other. However, from the proof of
Lemma \ref{BalancedTruncateLemma1}, it is clear that the resulting
reduced order system might be balanced, but the grammian $\Lambda_1$
of the reduced system is not necessarily the nice grammian. However,
from the proof of Lemma \ref{BalancedTruncateLemma1}, it follows that
$\Lambda_1$ is the nice observability and controllability grammian of
the system $\hat{\Sigma}=(r, Q, \{ (\hat{A}_q,\begin{bmatrix}
\hat{B}_q, & A_{q,12} \end{bmatrix}, \begin{bmatrix} \hat{C}_q \\
A_{q,21} \end{bmatrix} ) \mid q \in Q\})$.
It remains a topic of future research to find out if the balanced truncation
procedure can be adapted in such a way that $\Lambda_1$ remains a nice
grammian.
\end{Remark}

One may wonder if the system $\hat{\Sigma}$ returned by Procedure
\ref{BalancedTruncate} is minimal, at least when $\Sigma$ was
minimal. The answer is negative, as demonstrated by Example
\ref{example1}.
The fact that the reduced system need not even be minimal already
indicates that Procedure \ref{BalancedTruncate} might be too
conservative.

\begin{Example}
\label{example1}
Assume $Q=\{1\}$ consists of one element, $A=\begin{bmatrix} -2 & 0 &
0 \\ 0 & -1 & 1 \\ 0 & 0 & -3 \end{bmatrix}$, $B=\begin{bmatrix} 1 \\
0 \\ 1 \end{bmatrix}$, $C=\begin{bmatrix} 1 & 1 & 0 \end{bmatrix}$.
Then $(A,B,C)$ is balanced according to our definition with
$\Lambda=\diag(2,1,0.5)$.  However, $\hat{A}=\begin{bmatrix} -2 & 0 \\
0 &-1 \end{bmatrix}$, $\hat{B}=\begin{bmatrix} 1 \\ 0 \end{bmatrix}$
and $\hat{C}=\begin{bmatrix} 1 & 1 \end{bmatrix}$, which is clearly
not minimal.
\end{Example}

\begin{Theorem}[Error bound]
\label{DBalanceClaim8}
For the system $\hat{\Sigma}$ returned by Procedure
\ref{BalancedTruncate},
\begin{equation}
\label{DBalanceClaim8:eq1}
||Y^{\Sigma}-Y^{\hat{\Sigma}}||_{\mathcal{L}_2} \le 2 \sum_{k=r+1}^{n} \sigma_k.
\end{equation}
\end{Theorem}
\begin{proof}[Proof of Theorem \ref{DBalanceClaim8}]
The proof of Theorem \ref{DBalanceClaim8} is based in the following
lemma whose proof is in \ref{appsandberg}.
%
%
\begin{Lemma}
\label{BalanceClaim8.1}
For $r=n-1$, \eqref{DBalanceClaim8:eq1} is true.
\end{Lemma}

Suppose that $\hat{\Sigma}_1$ is the reduced system obtained by
removing the singular value $\sigma_n$.  It is easy to see that
$\hat{\Sigma}_1$ is again balanced with grammian $\Lambda_1$. We can
again apply the model reduction procedure to $\hat{\Sigma}_1$, remove
its smallest singular value $\sigma_{n-1}$ and obtain
$\hat{\Sigma}_2$. Suppose that the balanced system $\hat{\Sigma}_i$
with grammian $\Lambda_i=\diag(\sigma_1,\ldots,\sigma_{n-i})$ is
given. Define $\hat{\Sigma}_{i+1}$ as the system which is obtained
from $\hat{\Sigma}_i$ by applying the balanced truncation to the last
state, i.e., to the state which corresponds to $\sigma_{n-i}$.  In
this way, we obtain systems $\hat{\Sigma}_1,\ldots,
\hat{\Sigma}_{n-r}$ such that $\dim \hat{\Sigma}_i=n-i$ and
$||Y^{\hat{\Sigma}}_{i-1}-Y^{\hat{\Sigma}}_i||_{\mathcal{L}_2} \le 2
\sigma_{n-i+1}$, where $\hat{\Sigma}_0 = \Sigma$. Notice that
$\hat{\Sigma}_{n-r}=\hat{\Sigma}$
\[ ||\Sigma - \hat{\Sigma}||_{\mathcal{L}_2} \le \sum_{i=1}^{n-r}
||\hat{\Sigma}_{i-1} - \hat{\Sigma}_i||_{\mathcal{L}_2} \le
2\sum_{k=r+1}^{n} \sigma_{k},
\]
i.e., the error bound holds.
\end{proof}
\begin{Remark}
\label{min:suff}
Theorem \ref{DBalanceClaim8} and the second statement of
Theorem \ref{inv:theo} imply that it is enough to
apply balanced truncation to minimal \SLSS.
Indeed, let us apply balanced truncation to an \LSS\ $\Sigma$ with
controllability and observability grammians $(\CP,\Q)$. Assume that
$\sigma_1 \ge \ldots \ge \sigma_n$ are the singular values of $(\CP,\Q)$, $n=\dim \Sigma$.
Then the approximation error will be  bounded
by $2\sum_{i=r+1}^{n} \sigma_i$.
If we replace $(\CP,\Q)$ by the correspondig
grammians $(\CP_m,\Q_m)$ of a minimal \LSS\ $\Sigma_m$,
as described in Theorem \ref{inv:theo},
and we perform a balanced truncation by keeping the first $r$ singular values,
then the error bound becomes
$2\sum_{i=r+1}^{k} \lambda_i \le 2\sum_{i=r+1}^{n} \sigma_i$, i.e.
the error bound obtained by using the minimal \LSS\ does not exceed that of for the
original system.
\end{Remark}

\section{Relationship with Other Work}\label{RelationSection}
Results similar to those presented in this paper have already been
obtained by \cite{WoodGoddardGlover96,BeckThesis,Kotsalis}. More
precisely, \cite{WoodGoddardGlover96} studies the model reduction of a
linear parameter varying system,
\begin{eqnarray} \label{LPVsystem}
LPV:
\begin{cases}
 \dot x(t)= A_{q(t)} x(t) + B_{q(t)} u(t),~ x(0) = x_0 \\
y(t) = C_{q(t)} x(t), 
\end{cases}
\end{eqnarray}  
where $q$ is a continuous function $T \to \mathcal{B} \subset \mathbb
R^s$, $\mathcal{B} = [\underline{\rho}_1, \overline{\rho}_1] \times
\hdots \times [\underline{\rho_s}, \overline{\rho_s}]$ for some
$\underline{\rho}_i < \overline{\rho}_i$, and $A_q,B_q,C_q$ are
assumed to be continuous functions of $q$. In the following, we will
refer to the system \eqref{LPVsystem} as LPV.

The $L_2$ norm of the LPV system is
\[
||Y^{LPV}||_{L_2}=\sup_{q \in C(T, \mathcal{B})} \sup_{||u||_2=1}
||Y^{LPV}(u,q)||_{2}.
\]
Furthermore, \cite{WoodGoddardGlover96} uses similar LMI
characterization of the $L_2$ norm. Specifically, if $
||Y^{LPV}||_{L_2} < \gamma$ then there exists a solution $P
> 0$ to
\[\forall q \in \mathcal{B}~ : ~\mathbf G_{\gamma}(q, LPV, P) < 0.\]
The controllability and observability grammians in
\cite{WoodGoddardGlover96} are solutions $\mathcal Q$ and $\mathcal P$
to $\mathbf O(q,LPV, \mathcal Q) \le 0$ and $\mathbf C(q,LPV, \mathcal
P) \le 0$ for all $q \in \mathcal{B}$. Hence, they are also similar to
the continuous-time grammians used in this paper.


Despite apparent similarity between the LPV and LSS formulations, the two systems are not compatible as $q$ signal in \eqref{LPVsystem} is assumed to be a continuous function. 
Hence, the results of \cite{WoodGoddardGlover96} are not directly applicable to
linear switched systems.
Nonetheless, the proof of the error bounds in
\cite{WoodGoddardGlover96} does not use the continuity of the signal
$q$ nor the continuous dependence of $A_q,B_q,C_q$ on $q$; hence, the proof technique can be
adapted to the switched case.  As a result, the error bounds for the
balanced truncation provided in \cite{WoodGoddardGlover96} are
similar to \eqref{DBalanceClaim8:eq1}.  Note however, that the style of the proof
of Theorem \ref{DBalanceClaim8}
 is closer to that of \cite{SandbergThesis} than to \cite{WoodGoddardGlover96}.
 Moreover, counterparts of Theorem \ref{BalancClaim4} and Theorem \ref{BalancClaim7} and
Lemma \ref{BalancClaim1}, Lemma \ref{BalancClaim3} can be found in 
\cite{WoodGoddardGlover96}.
While
\cite{WoodGoddardGlover96} did state that their model reduction
procedure and the estimate of $L_2$ norm do not change under a
state-space isomorphism, the analysis in \cite{WoodGoddardGlover96} does not conclude that the
existence of LMI estimates of the $L_2$ norms, or indeed the existence and
singular values of the 
grammians is independent of the choice of state-space
realizations.  In fact, the latter would be difficult, since there seem to be no realization
theory for the type of LPV systems which is considered in \cite{WoodGoddardGlover96}.
Note that realization theory of certain classes of LPV systems was developed by
\cite{Toth1}.

In \cite{BeckThesis}, model reduction of uncertain discrete-time systems
was investigated. A \emph{structured uncertain system} was viewed as a Linear
Fractional Transformation (LFT),
\[ M \star \Delta = D + C \Delta (I - A \Delta)^{-1} B,\] where $\Delta:
l_2(\mathbb N,\mathbb R^n) \to l_2(\mathbb N,\mathbb R^n)$ represents
the uncertainty, and $M =  \left[
\begin{array}{c|c}
A &B \\ \hline
C &D
\end{array}
\right]$.  The norm of the system $M$ is the supremum of the
$H_{\infty}$ norms of $G \star \Delta$, where $\Delta$ is any element
of a bounded set of structured disturbances.

The model reduction procedure presented in \cite{BeckThesis} and the
corresponding error bounds are similar to the ones presented here for
discrete-time linear switched systems.  The main steps of the proofs
are also similar.  Nonetheless, the precise relationship between the
results of \cite{BeckThesis} and the ones presented above is not yet
clear.  In an attempt to clarify this connection, we represent a
linear switched system as a structured uncertain system as follows.

Let the structured uncertain system associated with a discrete-time
switched system $\Sigma=\SwitchSysLin$ be
\[
M_{\Sigma}=
\left[
\begin{array}{c|c}
A &B \\ \hline
C &D
\end{array}
\right]
\left[
\begin{array}{cccc|ccc}
0 & A_1 & \ldots & A_{\QNUM}   & B_1 & \ldots & B_{\QNUM}\\
I & 0 & \ldots & 0  & 0 & \ldots & 0 \\
\vdots & \vdots & \ldots & \vdots  & \ldots & \vdots \\
I & 0 & \ldots & 0  & 0 & \ldots & 0  \\
\hline
0 & C_1 & \ldots & C_\QNUM  & 0 & \ldots & 0 
\end{array}
\right]
\]
We fix an infinite switching sequence $v=q_0q_1 \cdots \in \mathcal{Q}$. 
For a $k\in \mathbb N$, we define the operators $d_q^v:l_2
(\mathbb N, \mathbb{R}^k) \rightarrow l_2(\mathbb N, \mathbb{R}^k)$ as
follows
$$d_q^v(z)(t)=\begin{cases}
              z(t)\chi_{\{\tau| q_{\tau} = q \}}(t-1) & \mbox{ if } t > 0  \\
             0  & \mbox{ if } t = 0.
 \end{cases}
$$
where
$\chi_{\{\tau| q_{\tau} = q \}}$ is the characteristic function of the
set $\{\tau \in \mathbb N|~ q_{\tau} = q \}$. 
Let $\delta^{-1}: l_2(\mathbb N, \mathbb{R}^n) \to \l_2(\mathbb N, \mathbb{R}^n),~ \delta^{-1}(z)(t)=z(t-1)$ for $t > 1$ and $\delta^{-1}(z)(0) = 0$, be the backward shift operator.
By abuse of notation, we
will apply the operators $d_q^v$ and $\delta^{-}$ to signals in Euclidean spaces of
different dimensions without specifying the dimension.

We define the uncertainty structure
\[ \Delta^v=\diag(\delta^{-1},d_1^{v}, \ldots, d_{\QNUM}^v). \] 

For an inputs $u\in l_2(\mathbb N,
\mathbb{R}^m)$, define $w_q = w_q(u) =d_q^v(\delta^{-1}(u))$ 
and let $x=(x_0,x_1,\ldots)$ be the state trajectory of $\Sigma$
which corresponds to the inputs $u$ and switching sequence $v$. If
$\Sigma$ is quadratically stable, then by Lemma \ref{BalancClaim6} $x \in l_2(\mathbb{R}^{n})$.
Define $z_q=\delta^{-1}(x)$. With this notation,

\[
\left[
\begin{array}{c}
x \\ z_1 \\ z_2 \\ \vdots \\ z_\QNUM \\ \hline
\delta^{-1}(y)
\end{array}
\right]
=
\left[
\begin{array}{c|c}
A\Delta^{v}  & B \\
\hline C\Delta^{v} & 0
\end{array}
\right]
\left[
\begin{array}{c}
x \\ z_1 \\ z_2 \\ \vdots \\ z_\QNUM \\ \hline w_1 \\ \vdots \\ w_{\QNUM}
\end{array}
\right],
\]
where $A,B,C$ are the corresponding sub-matrices of $M_{\Sigma}$. As a 
consequence, $\delta^{-1}(y)=(M_{\Sigma} \star \Delta^v)w$
with $w=w(u) = (w_1^T,\ldots,w_{\QNUM}^T)^T$.

We notice that the induced $l_2$ operator norm of $\Delta^v$ is $1$ for
any $v \in \mathcal{Q}$; hence, following the notion of \cite{BeckThesis}, $\Delta^v \in \mathbf B_{\Delta}$ with 
\[\mathbf B_{\Delta} = \{\Delta: l_2 (\mathbb N,
\mathbb{R}^{n(\QNUM+1)}) \to l_2 (\mathbb N,
\mathbb{R}^{n(\QNUM+1)})|~ ||\Delta|| \le 1 \}.\] Furthermore,
$l_2$-norm of $w = (w_1, \hdots, w_D)$ equals the $l_2$ norm of
$\delta^{-1}(u)$. 
In \cite{BeckThesis}, the norm $||M_{\Sigma}||$ of
$M_{\Sigma}$ is defined as follows
\[ ||M_{\Sigma}|| = \sup_{\Delta \in \mathbf{B}_{\Delta}} \sup_{w \in
  l_2(\mathbb{N}, \mathbb{R}^{m\QNUM}), ||w||_{2}=1} ||(M_{\Sigma}
\star \Delta)w||_2.
\]

We claim that $||Y^{\Sigma}||_{\mathcal{L}_2} \le ||M_{\Sigma}||$.
Indeed, from $\delta^{-1}(y)=(M_{\Sigma} \star \Delta^{v})w(u)$, it follows that
$||\delta^{-1}(y)||_2 \le ||M_{\Sigma} || \cdot ||w(u)||_{2}$.
Since $||w(u)||_{2}=||\delta^{-1}(u)||_{2}$, it then follows that
\( ||\delta^{-1}(y)||_2 \le ||M_{\Sigma} || \cdot ||\delta^{-1}(u)||_{2} \)
But $\delta^{-1}(y)$ is the response of $\Sigma$ to the input $\delta^{-1}(u)$ and 
the switching signal $\delta^{-1}(v)$.  Since the range of all possible choices of
$\delta^{-1}(u)$ and $\delta^{-1}(v)$ covers the whole space 
$\mathcal{U} \times \mathcal{Q}$, we get that
\[ \forall (u,q) \in \mathcal{U} \times \mathcal{Q}: ||Y^{\Sigma}(u)||_{2} \le
||M_{\Sigma}|| ||u||_{2} \] which implies 
that $||Y^{\Sigma}||_{\mathcal{L}_2} \le ||M_{\Sigma}||$.

Below, we state what we know about the relationship between the model
reduction procedure of this paper and that of \cite{BeckThesis}.
\begin{Lemma}
\label{beck:theo}
Let $\Sigma$ be a discrete-time \LSS,  and let $M_{\Sigma}$ be the
associated structured uncertain system. Then the following holds.
\begin{enumerate}
\item If $M_{\Sigma}$ is stable according to the
terminology of \cite{BeckThesis}, then $\Sigma$ is strongly stable.

\item The \LSS\ $\Sigma$ is minimal if and only if $M_{\Sigma}$ is
minimal according to the terminology of \cite{BeckThesis}.

\item Assume that the block-diagonal matrix $\CP=\mathrm{diag}(P_1,\ldots,P_{\QNUM+1})$, $0 <
P_{i} \in \mathbb{R}^{n \times n}$, $i=1,\ldots,\QNUM+1$
(resp. $\Q=\mathrm{diag}(Q_1,\ldots,Q_{\QNUM+1})$ $0 < Q_{i} \in
\mathbb{R}^{n \times n}$, $i=1,\ldots,\QNUM+1$) is a controllability
(resp. observability) grammian of $M_{\Sigma}$ according to the
terminology of \cite{BeckThesis}\footnote{Note that according to \cite{BeckThesis}, controllability and observability grammians are by definition block-diagonal}.

Then the following holds.
\begin{enumerate}
\item
$P_1$ (resp. $Q_1$) is a controllability
(resp. observability) grammian of $\Sigma$, and
\begin{equation}
\label{beck:theo:eq1}
\begin{split}
& \sum_{q \in Q} (A_{q}P_1 A^T_q+B_qB_q^T) - P_1 \le 0 \\
& \sum_{q \in Q} (A^T_{q}Q_1A_q+C^T_qC_q) - Q_1 \le 0. 
\end{split}
\end{equation}
\item If $M_{\Sigma}$ is balanced, i.e., $\CP=\Q$ is diagonal, then
$\Sigma$ is balanced.
\item If $M_{\Sigma}$ is balanced and $\hat{\Sigma}$ is the result of
applying Procedure \ref{BalancedTruncate} with the grammians
$P_1=Q_1$, then $M_{\hat{\Sigma}}$ is a result of applying balanced
truncation to $M_{\Sigma}$.
\end{enumerate}
\end{enumerate}
\end{Lemma}
The proof of Lemma \ref{beck:theo} is presented in \ref{appl10}.
One is tempted to try to use \cite{BeckThesis} for model reduction of
\SLSS. However, this leads to the following challenges.
\begin{itemize}
\item In order to apply the methods of \cite{BeckThesis}, $M_{\Sigma}$
has to be stable. By Lemma \ref{beck:theo}, stability of $M_{\Sigma}$
implies strong stability of $\Sigma$. Note that strong stability is a
more restrictive property than quadratic stability. Hence, the scope
of applicability of \cite{BeckThesis} appears to be smaller than that
of the current paper.

\item Even if $M_{\Sigma}$ is stable, we face restrictions. While by
Lemma \ref{beck:theo}, grammians of $M_{\Sigma}$ yield grammians of
$\Sigma$, it is not clear that the converse holds. Hence, the error
bound obtained by using \cite{BeckThesis} might be more conservative.

\item Let $\hat{M}$ be the result of balanced
truncation applied to $M_{\Sigma}$, as described in \cite{BeckThesis}.
Then $\hat{M}$  is a structured uncertain system, but it is not
clear how to convert the $\hat{M}$ to an \LSS. In fact, even
balancing might destroy the very specific structure of $M_{\Sigma}$
and hence make it difficult to interpret the balanced version of
$M_{\Sigma}$ as an \LSS.

\item Finally, while for balanced $M_{\Sigma}$, the results of
Procedure \ref{BalancedTruncate} and the procedure from
\cite{BeckThesis} are comparable, it is not very clear how these two
procedures are related in the general case.
\end{itemize}
Despite the difficulties mentioned above, exploring the relationship with
\cite{BeckThesis} remains worthwhile. In particular, the results of
Lemma \ref{beck:theo} indicate that the relationship might be much closer than
it appears at the first sight.
Intuitively, it is also clear why \cite{BeckThesis} seems to yield more conservative
results: the behavior of an \LSS\ corresponds to a \emph{subset} of behaviors of
a structured uncertain system.
Hence, the model reduction procedure of \cite{BeckThesis} has to preserve a much richer
behavior than the one presented in this paper.
To sum up, despite numerous similarities, it is unclear if
\cite{BeckThesis} can be used for model reduction of \SLSS.

Concerning the work of \cite{Kotsalis}, the main difference is that we
consider deterministic systems, while \cite{Kotsalis} considers
stochastic systems with switching modeled as a Markov process on
$Q$. The nice grammians of the present paper correspond to the
grammians of \cite{Kotsalis}, if we associate with the discrete-time
deterministic $\Sigma$ the following stochastic system
\[
\Sigma_{\mathrm{st}} : \left\{
\begin{split}
& \widetilde{x}(t+1)=\frac{1}{\sqrt{p}}A_{\theta(t)}\widetilde{x}(t)+\frac{1}{\sqrt{p}}B_{\theta(t)}u(t) \\
& \widetilde{y}(t)=\frac{1}{\sqrt{p}}C_{\theta(t)}\widetilde{x}(t),
\end{split}\right.
\]
where $\theta(t) \in Q$ is an identically distributed independent
process, $u(t)$ is deterministic, $x_0=0$, and $p=P(\theta(t)=q) > 0$
for all $q \in Q$.
In \cite{Kotsalis}, the norm of the system is smaller than $\gamma$ if
$\sum_{t=0}^{\infty} E[||\widetilde{y}(t)||^2] \le \gamma^2
||u||^2_{2}$. 
To ensure that the system norm is finite, the stochastic systems at hand
are assumed to be mean-square stable and only inputs $u \in
l_2(\mathbb{N},\mathbb{R}^m)$ are considered.  With the correspondence
above, the balancing procedure in our work becomes similar to that of
\cite{Kotsalis}. This is summarized in the following lemma.
\begin{Lemma}
\label{kotsalis:theo}
Consider the discrete-time \LSS\ $\Sigma$ and let $\Sigma_{\mathrm{st}}$ be the
associated stochastic system. Then the following holds.
\begin{enumerate}
\item
 \label{kotsalis:theo:part1}

 If $\Sigma$ is strongly stable if and only if $\Sigma_{\mathrm{st}}$ is
mean-square stable according to \cite{Kotsalis}.
\item 
 \label{kotsalis:theo:part2}

$\CP$ (resp. $\Q$) is a controllability (resp. observability)
grammian of $\Sigma_{\mathrm{st}}$ according to the terminology of
\cite{Kotsalis} if and only if 
\begin{equation}
\label{kotsalis:theo:eq1}
\begin{split}
& \sum_{q \in Q} (A_{q}\CP A^T_q+B_qB_q^T) - \CP \le 0 \\
& \sum_{q \in Q} (A^T_{q}\Q A_q+C^T_qC_q) - \Q \le 0. \\
\end{split}
\end{equation}
In particular, controllability and observability grammians of
$\Sigma_{\mathrm{st}}$ are controllability and observability grammians
of $\Sigma$. Conversely, nice controllability and nice observability
grammians of $\Sigma$ are controllability and observability grammians
of $\Sigma_{\mathrm{st}}$.
\item 
 \label{kotsalis:theo:part3}

The balanced reduction algorithm presented in \cite{Kotsalis} coincides
with Procedure \ref{BalancedTruncate}, if the latter is applied to
grammians of \eqref{kotsalis:theo}.
\item 
 \label{kotsalis:theo:part4}

If the norm of $\Sigma_{st}$ is $\gamma$ according to
\cite{Kotsalis}, then $||Y^{\Sigma}||_{l_2} \le \gamma$. 
\end{enumerate}
\end{Lemma}
The proof of Lemma \ref{kotsalis:theo} is presented in \ref{appl11}.
From Lemma \ref{kotsalis:theo}, it follows that the error bound for
the balanced truncation in \eqref{DBalanceClaim8:eq1} follows from the
error bound derived in \cite{Kotsalis}, if one uses nice
grammians. However, in \cite{Kotsalis} the questions related to
minimality and dependence of the grammians on state-space realization
were not discussed. Note that the results \cite{Kotsalis} are
directly applicable only to strongly stable linear switched systems,
while the results of the current paper are formulated for
quadratically stable \SLSS. 
Furthermore, note that \cite{Kotsalis} is applicable only to a subset of grammians. For this
reason, the model reduction procedure from \cite{Kotsalis}, when
applied to deterministic \SLSS\ via the embedding above,  is likely
to yield a more conservative error bound.
This is not suprising, since 
\cite{Kotsalis} addresses model reduction of
stochastic systems, of which deterministic systems form a subclass. 


\bibliographystyle{elsarticle-num} 

\appendix

\section{Technical proofs}\label{app}

For some of the proofs below we will need the following simple
consequence of using Schur complements.
\begin{Lemma}
\label{Claim0}
Let $A,P \in \mathbb{R}^{n \times n}$, $S \in \mathbb{R}^{k \times
  n}$, $P > 0$. Let $\epsilon \in \{<,\le\}$ Then the following holds.
\begin{enumerate}
\item $A^TP+PA+S^TS\  \epsilon\  0$, if and only if $AP^{-1} +P^{-1}A^T+P^{-1}S^TSP^{-1}
\ \epsilon\ 0$, or, equivalently
\[
\begin{bmatrix}
P^{-1}A^T+AP^{-1} & P^{-1}S^T \\
SP^{-1} & -I
\end{bmatrix} \epsilon 0.
\]
\item $-P+A^TPA+S^TS\ \epsilon\ 0$, if and only if
\[ \begin{bmatrix}
-P^{-1}+AP^{-1}A^T & -AP^{-1}S^T \\
-SAP^{-1} & -I+SP^{-1}S^T
\end{bmatrix}\ \epsilon\ 0.
\]
\end{enumerate}
\end{Lemma}
\begin{proof}[Proof of Lemma \ref{Claim0}]
The first statement follows by multiplying $A^TP+PA+S^TS$ by $P^{-1}$
from left and right and noticing that since $P^{-1}$ is symmetric,
$A^TP+PA+S^TS \epsilon 0$ is equivalent to $P^{-1}(A^TP+PA+S^TS)P^{-1} \epsilon  0$.  By
taking Schur complements, it follows that $AP^{-1}
+P^{-1}A^T+P^{-1}S^TSP^{-1}\ \epsilon\ 0$ is equivalent to
\[
\begin{bmatrix}
P^{-1}A^T+AP^{-1} &  P^{-1}S^T \\
SP^{-1} & -I
\end{bmatrix}\ \epsilon\ 0.
\]

In order to prove the second statement, the discrete case, notice that
$-P+A^TPA+S^TS\ \epsilon\  0$ is equivalent to
\[
-(P- \begin{bmatrix} A \\ S
\end{bmatrix}^T \begin{bmatrix} P & 0 \\ 0 & I_{k}
\end{bmatrix}
\begin{bmatrix}
A \\ S
\end{bmatrix})\ \epsilon\ 0.
\]
Using Schur complements again, the latter is equivalent to
\[
- \begin{bmatrix}
P & A^T & S^T \\
A & P^{-1} & 0 \\
S & 0 & I_k
\end{bmatrix}\ \epsilon\ 0
\]
Multiplying the latter inequality by $\mathcal{S}=\begin{bmatrix} 0 & I_{k+n} \\
I_{n} & 0 \end{bmatrix}$ from right and by $\mathcal{S}^T$ from left, we get
the following equivalent LMI
\[
-\begin{bmatrix}
P^{-1} & 0 & A \\
0 & I_k & S \\
A^T & S^T & P
\end{bmatrix}\ \epsilon\ 0.
\]
By using Schur complement again, from this we obtain that the latter
LMI is equivalent to
\[ \begin{bmatrix}
-P^{-1}+AP^{-1}A^T &  -AP^{-1}S^T \\
-SP^{-1}A^T & -I_k +SP^{-1}S^T
\end{bmatrix}\ \epsilon\ 0.
\]
\end{proof}

\section{}\label{appl1}

\begin{proof}[Proof of Lemma \ref{BalancClaim9}]
\textbf{(ii) implies (i)}
  It is clear that if $\forall q \in Q: \mathbf{O}(q,\Sigma,\Q) < 0$, then
  $\forall q \in Q: \mathbf{S}(q,\Sigma,\Q) < 0$.

\textbf{(iii) implies (i)}
   If $\forall q \in Q: \mathbf{C}(q,\Sigma,\CP) < 0$, then 
   by taking $A=A_q^T$ and $S=B_qB_q^T$, it follows from Lemma
   \ref{Claim0} that 
  $\forall q \in Q: \mathbf{S}(q,\Sigma,\CP^{-1}) < 0$.

\textbf{(i) implies (ii) and (iii)}.
We present the proof separately for the discrete-time and for the
continuous-time case.
   
\textbf{Continuous-time:} Assume that for some $P > 0$, \( \forall q
\in Q: A^T_qP+PA_q < 0. \) Then for any $q \in Q$, the exists a scalar
$\gamma_q > 0$ such that \( A^T_qP+PA_q+\gamma_q M_q < 0. \) Take
$M_q=C_q^TC_q$ and let $\gamma = \min\{\gamma_q \mid q \in Q\}$.  Then
\( A^T_qP+PA_q+\gamma M_q < 0. \) Define $\mathcal{S}=\frac{1}{\gamma}
P$. Then
\[ \forall q \in Q: A^T_q\mathcal{S}+\mathcal{S} A_q + M_q < 0. \]
By choosing $M$ to be $C_q^TC_q$ and $\Q=\mathcal{S}$, we obtain a
solution to \eqref{ContObsGram}.  If we choose $M_q=PB_qB_q^TP$, then
using the fact that $\mathcal{S}=\frac{1}{\gamma}P$ and Lemma
\ref{Claim0}, we get that
\[ \forall q \in Q:
A_q\mathcal{S}^{-1}+\mathcal{S}^{-1}A^T_q+B_qB_q^T\gamma^{2} < 0. \] By
choosing
$\CP=\frac{1}{\gamma^{2}}\mathcal{S}^{-1}=\frac{1}{\gamma}P^{-1}$ and using
Schur complements, we
get that \eqref{ContContrGram} holds.

\textbf{Discrete-time:} If $P > 0$ is such that \( P-A_{q}^TPA_q >
0\), then for any $M_q \ge 0$, there exists $\gamma_q > 0$ such that
\( P-A_q^TPA_q -\gamma_q M_q > 0\). In particular, by taking
$\gamma=\min\{\gamma_q \mid q \in Q\}$, \( P -A_q^TPA_q -\gamma M_q >
0 \), or, in other words
\[ \forall q \in Q: A^T_q\mathcal{S}A_q + M_q - \mathcal{S} <  0, \] where
$\mathcal{S}=\frac{1}{\gamma} P$.  If we choose $M_q=C_qC_q^T$ and set
$\Q=\mathcal{S}$, then we get that \eqref{ContObsGram} holds.  From
the second part of Lemma \ref{Claim0} it follows that if
$P-A_{q}^TPA_q > 0$, then $P^{-1}-A_{q}P^{-1}A^T_q > 0$. By
interchanging $A_q$ and $A_q^T$ and using $P^{-1}$ instead of $P$, we
can repeat the argument above. We thus get that for any $M_q \ge 0$,
there exists $\CP > 0$ such that
\[ \forall q \in Q: A_q\CP A^T_q + M_q - \CP < 0. \] By taking
$M_q=B_qB_q^T$, it follows that \eqref{ContContrGram} holds.
\end{proof}

\section{}\label{appl2}

\begin{proof}[Proof of Lemma \ref{BalancClaim6}]

In continuous-time, the proof that existence of a solution $P > 0$ to
\eqref{BalancClaim6:eq1} implies that $||Y^{\Sigma}||_{\mathcal{L}_2}$ exists
and $||Y^{\Sigma}||_{\mathcal{L}_2} \le \gamma$ follows from \cite[Theorem
1]{HirataHespanhaDec09} by taking $V(x)=x^T P x$. A similar argument
can be done for the discrete-time case.

In order to be self-contained, we present an elementary proof of the
implication \[ \forall q \in Q: \mathbf{G}_{\gamma}(q,P,\Sigma) <0
\implies ||Y^{\Sigma}||_{\mathcal{L}_2} < \gamma \] both for the continuous- and
discrete-time case.

Fix an input and switching signal $(u,q) \in \mathcal{U} \times
\mathcal{Q}$ and denote by $x$ and $y$ the corresponding state and
output trajectory of $\Sigma$.  We have to show that for $u \in
\mathcal{U}$, the output $y$ belongs to $L_2(T,\mathbb{R}^{p})$ for
continuous-time case and to $l_2(\mathbb{N},\mathbb{R}^p)$ for the
discrete-time case.

Notice that if we define 
\begin{align*}
||y||_{2}=
\begin{cases}
\sqrt{\lim_{t \rightarrow \infty}
  \int_0^{t} ||y(s)||_{2}^{2}ds} &\text{cont.}\\
\sqrt{\lim_{t \rightarrow \infty} \sum_{s=0}^{t} ||y(s)
  ||_{2}^{2}ds}&\text{disc.}\\
\end{cases}
\end{align*}
then $||y||_{2}$ is well-defined
(possibly equal $+\infty$) and $y$ belongs to the $L_2(T,\mathbb{R}^p)$ 
(resp. $l_2(\mathbb{N},\mathbb{R}^{p})$)
if and only if $||y||_{2} <+\infty$. In the latter case, $||y||_{2}$ is just the
standard $L_2$ and $l_2$ norm respectively.
 Hence, it is enough to
show $||y||_{2}^{2} \le \gamma^{2}||u||_{2}^{2}$, if we use the extended definition of
$||y||_2$ described above.

Assume $P$ is a solution to \eqref{BalancClaim6:eq1}. Define
\begin{align*}
\Delta(x^T(t)P&x(t))\\
&=
\begin{cases}
\frac{d}{dt} (x^T(t)Px(t)) \mbox{\ \  (cont.) } \\
x^T(t+1)Px(t+1)-x^T(t)Px(t) \mbox{\ \ (disc.) }
\end{cases}
\end{align*}
Then a simple calculation reveals that both for continuous and
discrete-time cases,
\begin{align*}
\Delta(x^T(t)Px(t)) &= \
\begin{bmatrix} x(t) \\ u(t) \end{bmatrix}^T G_{\gamma}(q(t),P,\Sigma)
\begin{bmatrix} x(t) \\ u(t) \end{bmatrix}\\
&\quad +\gamma^2||u(t)||_{2}^2- x(t)^TC_{q(t)}^TC_{q(t)}x(t) \\
&\le \gamma^2||u(t)||^2-||y(t)||^2
\end{align*}
Notice that for continuous-time systems 
\begin{align*}
\int_0^{t}\Delta(x^T(s)Px(s))&=x^T(t)Px(t)-x^T(0)Px(0)\\
&=x^T(t)Px(t).  
\end{align*}
Similarly,
for the discrete-time systems, 
\begin{align*}
\sum_{s=0}^{t-1}\Delta(x^T(s)Px(s))&=x^T(t)Px(t)-x^T(0)Px(0)\\
&=x^T(t)Px(t).  
\end{align*}
In both
cases, we used the fact that $x(0)=0$ (see page \pageref{zero}). Recall that for
continuous-time case 
\begin{align*}
&||u||_{2}^{2} = \int_{0}^{\infty} ||u(s)||_{2}^{2}ds \ge \int_{0}^{t}
||u(s)||_{2}^{2}ds.  
\end{align*}
Similarly, for the discrete-time case,
\begin{align*}
&||u||_{2}^{2} =
\sum_{s=0}^{\infty} ||u(s)||_{2}^{2} \ge \sum_{s=0}^{t-1}
||u(s)||_{2}^{2}.  
\end{align*}
Hence, by taking integral in the continuous-time
case and sums in the discrete-time case, we obtain
\[
\begin{split}
& x^T(t)Px(t) \le \gamma^2 ||u||^2_2-\begin{cases} \int_0^{t} ||y(s)||_2^{2}ds & \mbox{ (cont.)} \\
 \sum_{s=0}^{t-1} ||y(s)||_2^2 & \mbox{(disc.)} 
 \end{cases}.
\end{split}
\]
Since $P >0$, $x^T(t)Px(t) \ge 0$ and thus the inequality above yields
\[ \begin{cases} \int_0^{t} ||y(s)||_2^{2}ds & \mbox{ (cont.)} \\
 \sum_{s=0}^{t-1} ||y(s)||_2^2 & \mbox{(disc.)}
 \end{cases}
 \le \gamma^2 ||u||^2_2. 
\]
By taking limit of the left-hand side as $t \rightarrow \infty$, it follows that
$||y||^2_2 \le \gamma^{2} ||u||_2^2$.

If $\Sigma$ is quadratically stable, then by Lemma \ref{BalancClaim9}, 
there exists $\CP > 0$ such that $\forall q \in Q: \mathbf{C}(q,\Sigma,\CP) < 0$.
By taking $A=A^T_q$, $S=B^T_q$ in
Lemma \ref{Claim0}, it then follows that for all $q \in Q$,
\[ G_{1}(q,\Sigma,\CP^{-1}) - \begin{bmatrix} C_q^TC_q & 0 \\
0 & 0 \end{bmatrix} < 0. \] Since $C_q^TC_q \ge 0$, it then follows
that there exists a large enough $\gamma > 0$ such that
\begin{equation}
\label{BalancClaim6:pf:eq1}
G_{1}(q,\Sigma,\CP^{-1}) +
(\frac{1}{\gamma^{2}}-1)\begin{bmatrix} C_q^TC_q & 0 \\ 0 &
0 \end{bmatrix} < 0.  
\end{equation}
Notice now that
\[ \gamma^2 G_{1}(q,\Sigma,\CP^{-1})= G_{\gamma}(q,\Sigma,
\gamma^2\CP^{-1})+ (\gamma^{2}-1)\begin{bmatrix} C_q^TC_q & 0
\\ 0 & 0 \end{bmatrix}
\]
By multiplying \eqref{BalancClaim6:pf:eq1} with $\gamma^2$ and using the equality
above,
\[
\begin{split}
0 &> G_{\gamma}(q,\Sigma, \gamma^2\CP^{-1})+
(\gamma^{2}-1)\begin{bmatrix} C_q^TC_q & 0 \\ 0 & 0 \end{bmatrix}\\
&\quad 
+(1-\gamma^{2})\begin{bmatrix} C_q^TC_q & 0 \\ 0 &
0 \end{bmatrix}\\ 
&=G_{\gamma}(q,\Sigma,\gamma^2\CP^{-1}).
\end{split}
\]
Hence, $\gamma^{2}\CP^{-1}$ satisfies \eqref{BalancClaim6:eq1}.
\end{proof}

\section{}\label{appl3}

\begin{proof}[Proof of Lemma \ref{inv:theo:lemma2}]
In the proof we will used the notation \eqref{LSSreach:eq1} of
Procedure \ref{LSSreach}.  The lemma follows from the following
observations.\\
\textbf{Observation 1}
\[
A^T_{q}P+PA_{q} = \begin{bmatrix}
(A_q^{R})^TP_{11} +P_{11}A_q^R & \star \\
\star & \star
\end{bmatrix}
\]
\textbf{Observation 2}
\[
\begin{split}
A_q^TPA_q &= \begin{bmatrix} (A_q^R)^T P_{11} & \star \\
\star & \star
\end{bmatrix}
\begin{bmatrix} A_{q}^{\mathrm R} & A^{'}_{q} \\ 0 & A^{''}_{q} \end{bmatrix} \\
& = \begin{bmatrix} (A_q^R)^T P_{11}A_q^R & \star \\
\star & \star
\end{bmatrix}
\end{split}
\]
\textbf{Observation 3}
\[ C_q^TC_q = \begin{bmatrix} (C_q^R)^TC_q^R & \star \\
\star & \star
\end{bmatrix}
\]
\textbf{Observation 4}
\[
PB_q = \begin{bmatrix} P_{11}B_q \\ \star \end{bmatrix}.
\]
\textbf{Observation 5}
\[
B_q^TPB_q = (B_{q}^R)^TP_{11}B^R_q.
\]

If $\mathbf{K}=\mathbf{S}$, then \textbf{Observation 1} implies the
statement of the lemma for the continuous-time case and \textbf{Observation 2}
implies the statement of the lemma for discrete-time case.

Finally, by combining \textbf{Observation 3}, \textbf{Observation 1}
and \textbf{Observation 4}, it follows that for continuous time 
\begin{equation}
\label{inv:theo:lemma2:eq1}
\begin{split}
\mathbf{G}_{\gamma}&(q,\Sigma,P)\\
&=
\begin{bmatrix} (A_{q}^R)^TP_{11}+P_{11}A_{q}^R+(C_q^R)^TC^R_q  & \star &  P_{11}B^R_q \\
\star                                          & \star & \star  \\
(B_q^R)^TP_{11} & \star & -\gamma^2I
\end{bmatrix}.
\end{split}
\end{equation}

It is easy to see that \( \mathbf{G}_{\gamma}(q,\Sigma,P) < 0 \)
implies that
\[
\begin{bmatrix} (A_{q}^R)^TP_{11}+P_{11}A_{q}^R+(C_q^R)^TC^R_q  & P_{11}B^R_q \\
(B_q^R)^TP_{11} & -\gamma^2I
\end{bmatrix} < 0,
\]
and the latter is equivalent to
$\mathbf{G}_{\gamma}(q,\Sigma_{R},P_{11}) < 0$.  By combining
\textbf{Observation 3}, \textbf{Observation 2} and \textbf{Observation
  5} for discrete-time case, we obtain that
\begin{equation}
\label{inv:theo:lemma2:eq2}
\begin{split}
\mathbf{G}_{\gamma}&(q,\Sigma,P)\\
&=
\begin{bmatrix} (A_{q}^R)^TP_{11}A_{q}^R+(C_q^R)^TC^R_q & \star &
(A_{q}^{R})^TP_{11}^TB^R_q \\
\star                                          & \star & \star  \\
A_{q}^{R}P_{11}B^R_{q} & \star & B_{q}^{R}P_{11}(B_{q}^R)^T-\gamma^2I
\end{bmatrix}.
\end{split}
\end{equation}

Hence, by a similar argument as for the continuous-time case,
$\mathbf{G}_{\gamma}(q,\Sigma,P)$ $ < 0$ implies
$\mathbf{G}_{\gamma}(q,\Sigma_r,P_{11}) < 0$.

By combining \textbf{Observation 3}, \textbf{Observation 2}, and 
\textbf{Observation 1}, it follows
that
\[
  \mathbf{O}(q,\Sigma,P)=\begin{bmatrix} \mathbf{O}(q,\Sigma_r,P_{11}) & \star \\
                                          \star & \star 
                          \end{bmatrix},
\]
and hence $\mathbf{O}(q,\Sigma,P) \le 0$ implies $\mathbf{O}(q,\Sigma_r,P_{11}) \le 0$.

Finally, we will show that if $\forall q \in Q: \mathbf{C}(q,\Sigma,P^{-1}) \le 0$,
then $\forall q \in Q: \mathbf{C}(q,\Sigma_r,P_{11}^{-1}) \le 0$.
From Lemma \ref{Claim0} it follows that
$\forall q \in Q: \mathbf{C}(q,\Sigma_r,P^{-1}) \le 0$ if and only if
$\forall q \in Q: \mathbf{G}_{1}(q,\Sigma,P)-C_q^TC_q \le 0$.
From \textbf{Observation 3} and \eqref{inv:theo:lemma2:eq1} -- \eqref{inv:theo:lemma2:eq2} it follows that
$\forall q \in Q: \mathbf{G}_{1}(q,\Sigma,P)-C_q^TC_q \le 0$ implies 
$\forall q \in Q: \mathbf{G}_{1}(q,\Sigma_r,P_{11})-(C^R_q)^TC^R_q \le 0$.
From Lemma \ref{Claim0} it then follows that
$\mathbf{C}(q,\Sigma_r,P_{11}^{-1}) \le 0$ for all $q \in Q$.
\end{proof}

\section{}\label{appl4}

\begin{proof}[Proof Lemma \ref{inv:theo:lemma3}]
The first part of the statement follows directly from Lemma
\ref{Claim0} by taking $S=0$.
The second statement follows by noticing that $P \in \mathbf{O}(\Sigma) \iff P \in \mathbf{C}(\Sigma^T)$ and $P \in \mathbf{C}(\Sigma) \iff P \in \mathbf{O}(\Sigma^T)$.
The third statement can be seen as
follows.  For the continuous-time case, notice that $P \in
G_{\gamma}(\Sigma)$ is equivalent to
\begin{equation}
\label{inv:theo:lemma3:eq1}
\forall q \in Q: A_{q}^TP+PA_q+C_q^TC_q + \frac{1}{\gamma^2} PB_qB_q^TP < 0. 
\end{equation}
By applying Lemma \ref{Claim0} to \eqref{inv:theo:lemma3:eq1} it then
follows that \eqref{inv:theo:lemma3:eq1} is equivalent to
\begin{equation}
\label{inv:theo:lemma3:eq2}
\forall q \in Q: P^{-1}A_{q}^T+A_qP^{-1}+P^{-1}C_q^TC_qP^{-1} + \frac{1}{\gamma^2} B_qB_q^T < 0. 
\end{equation}
If we multiply \eqref{inv:theo:lemma3:eq2} by $\gamma^{2}$, we
immediately get that $R=\gamma^{2}P^{-1}$ satisfies
\[ \forall q \in Q:
A_qR+RA^T_q+B_qB_q^T+\frac{1}{\gamma^{2}}RC_q^TC_qR, \] and the latter
is equivalent to $R \in G_{\gamma}(\Sigma^T)$.

For the discrete-time case, notice that $P \in G_{\gamma}(\Sigma)$ is
equivalent to 
\begin{equation}
\label{inv:theo:lemma3:eq3}
\forall q \in Q: \begin{bmatrix}
-\frac{1}{\gamma^2}P+A_q^T\frac{1}{\gamma^{2}}PA_q+\frac{1}{\gamma^2}C_q^TC_q
&
A_q^T\frac{1}{\gamma^{2}}PB_q \\
(A_q^T\frac{1}{\gamma^{2}}PB_q)^T & B_q^T\frac{1}{\gamma^2}PB_q - I
\end{bmatrix} < 0
\end{equation}
From \eqref{inv:theo:lemma3:eq3} it follows that
\begin{equation}
\label{inv:theo:lemma3:eq4}
\forall q \in Q: \hat{P}-\hat{A}^T_q\hat{P}\hat{A}_q  > 0 
\end{equation}
where
\[ \hat{A}_q = \begin{bmatrix} A_{q} & B_q \\
\frac{1}{\gamma} C_q & 0
\end{bmatrix} \mbox{ and }
\hat{P}=\begin{bmatrix} \frac{1}{\gamma^2}P & 0 \\
0 & I
\end{bmatrix}.
\]
Applying the discrete-time part of Lemma \ref{Claim0} with $S=0$,
$P=\hat{P}$ and $A=\hat{A}_q$, we get that \eqref{inv:theo:lemma3:eq4}
is equivalent to
\begin{equation}
\label{inv:theo:lemma3:eq5}
\forall q \in Q: \hat{P}^{-1} - \hat{A}_q\hat{P}^{-1}\hat{A}_q^T > 0. 
\end{equation}
Using the definition of $\hat{A}_q$ and the fact that
$\hat{P}^{-1}=\mathrm{diag}(\gamma^{2}P^{-1}, I)$, it follows that
\eqref{inv:theo:lemma3:eq5} is equivalent to
\begin{equation}
\label{inv:theo:lemma3:eq6}
\forall q \in Q: \begin{bmatrix}
-R+A_qRA^T_q+B_q^TB_q &
\frac{1}{\gamma}A_q^TRC_q \\
(\frac{1}{\gamma}A_q^TRC_q)^T &  \frac{1}{\gamma^2}C_q^TRC_q - I
\end{bmatrix} < 0,
\end{equation}
where $R=\gamma^{2}P^{-1}$.  By multiplying
\eqref{inv:theo:lemma3:eq6} by $\diag(I_{n},\gamma I_p)$ from left and
right, it follows that \eqref{inv:theo:lemma3:eq6} is equivalent to
$\gamma^{2}P^{-1} \in G_{\gamma}(\Sigma^T)$.
\end{proof}

\section{}\label{appl5}

\begin{proof}[Proof of Lemma \ref{BalancClaim1}]
Fix a switching signal $q \in \mathcal{Q}$ and denote by $x(t)$ and $y(t)$ the state
trajectory of $\Sigma$ such that $x(0)=x$ and $u=0$, i.e.
$x=X^{\Sigma}_0(0,q)$ and $y=Y^{\Sigma}_0(0,q)$.
Define
\begin{align*}
\Delta(x^T(t)&\Q x(t))\\
&=
\begin{cases}
\frac{d}{dt} (x^T(t)\Q x(t)) & \mbox{\ \  (cont.) } \\
x^T(t+1)\Q x(t+1)-x^T(t)\Q x(t) & \mbox{\ \ (disc.) }
\end{cases}
\end{align*}
and denote
\[ \mathbf{S}(q,\Q)=\begin{cases}
A_q^T\Q +\Q A_q & \mbox{\ \  (cont.) } \\
A_q^T\Q A_q & \mbox{\ \ (disc.) }
\end{cases}.
\]
Then $\Q$ satisfies $\forall q \in Q: \mathbf{S}(q,\Q)\le - C_q^TC_q$
and \[\Delta(x^T(t)\Q x(t))=x^T(t)\mathbf{S}(q(t),\Q)x(t).\] Hence, it
follows that
\begin{equation}
\label{BalancClaim1:pf:eq1} 
\Delta(x^T(t)\Q x(t)) \le -x^T(t)C^T_{q(t)}C_{q(t)}x(t)=-||y(t)||_{2}^{2}. 
\end{equation}
Notice that
\[
x^T(t)\Q x(t) - x^T\Q x =
\begin{cases}
\int_0^t \Delta(x^T(s)\Q x(s))ds &\mbox{\ \  (cont.) } \\
\sum_{s=0}^{t-1} \Delta(x^T(t)\Q x(t)) &\mbox{\ \ (disc.) }
\end{cases}
\]
and that $x^T(t)\Q x(t) \ge 0$ and hence $-x^T\Q x \le x^T(t)\Q x(t) -
x^T\Q x$.  By taking integrals $\int_0^{t} ||y(s)||_{2}^2ds$ in the
continuous-time case and sums $\sum_{s=0}^{t-1} ||y(s)||_{2}^{2}$ in
the discrete-time case, and using \eqref{BalancClaim1:pf:eq1} it
follows that
\[ -x^T\Q x \le
\begin{cases}
-\int_{0}^{t} ||y(s)||^2_{2}ds &\text{(cont.)}\\
-\sum_{k=0}^{t} ||y(s)||^2_{2} &\text{(disc.)}
\end{cases}.
\]
By multiplying the inequality above by $-1$ the statement of the lemma
follows.
\end{proof}

\section{}\label{appl6}

\begin{proof}[Proof of Lemma \ref{BalancClaim3}]
Denote by $\Sigma_0$ the \BLSS $\Sigma_0$ obtained from $\Sigma$ by
replacing $C_q$, $q \in Q$ by zero matrices.  For any input $u \in \mathcal{U}$ and
switching signal $q \mathcal{Q}$, the state trajectory $x(t)$ of $\Sigma_0$ and
$\Sigma$ are the same, but the output trajectory $y_0$ of $\Sigma_0$ is
identically zero.  By taking $A=A_q^T$, $S=B_q^T$, $ q\in Q$, from
Lemma \ref{Claim0} it follows that $\CP^{-1}$ satisfies
\[ \forall q \in Q: \mathbf{G}_{1}(q,\CP^{-1},\Sigma_0) \le 0. \] Hence,
from the proof of Lemma \ref{BalancClaim6} (when applied to $\Sigma_0$
instead of $\Sigma$) it follows that
\[ x^T(t)\CP^{-1}x(t) \le ||u||_{2}^{2}-||y_0||_{2}^{2} =
||u||_{2}^{2}. \]
\end{proof}

\section{}\label{appl7}

\begin{proof}[Proof of Lemma \ref{nice_stab:proof}]
One can easily see that the matrix $\widetilde{A}=\sum_{q \in Q} F_q^T
\otimes F_q^T$ is in fact a matrix representation of the linear map
$\Z : \mathbb{R}^{n\times n} \to \mathbb{R}^{n\times n}$ defined as \[
\Z(V)=\sum_{q \in Q} F_q^TVF_q .\] This result is obtained by
identifying $\mathbb{R}^{n\times n}$ with $\mathbb{R}^{n^2}$, as it is
done in \cite[Section 2.1]{CostaBook}. As a consequence, the
eigenvalues of $\Z$ and $\widetilde{A}$ coincide. Since the
eigenvalues of $\Z$ are inside the unit circle, it follows from
\cite[Proposition 2.6]{CostaBook} that $P=\Z(P)+Q$ has a unique
solution. Notice that using the terminology of \cite[page
17]{CostaBook} $\Z$ is a Hermitian map and positive operator. Indeed,
if $V$ is symmetric, then so is $\Z(V)$ and if $V$ is positive
semi-definite, then so is $\Z(V)$. Hence, by \cite[Proposition
2.6]{CostaBook} the solution of $P=\Z(P)+\mathcal{G}$ is positive semi-definite,
and if $\mathcal{G} > 0$, then $P$ is positive definite.  Moreover, notice that
$\Z^{k}(\mathcal{G})=\sum_{w \in Q^{*}, |w|=k} F_{w}^T\mathcal{G}F_{w}$ and hence by
\cite[Proposition 2.6]{CostaBook}, the solution $P=\sum_{k=0}^{\infty}
\Z^k(\mathcal{G})=\sum_{w \in Q^{*}} F_{w}^T\mathcal{G}F_w$.

Conversely, assume that $P-\sum_{q \in Q} A^T_{q}PA_q > 0$ for some $P
> 0$.  Consider an $n \times n$ matrix $V$ such that $V \ge 0$.
Define the map $\K(V)=\sum_{q \in Q} A_qVA_q^T$ and notice that using
the coordinate representation of \cite[page 17]{CostaBook},
$\widetilde{A}^T=\sum_{q \in Q} F_q \otimes F_q$ is the matrix
representation of $\K$. Since taking transposes preserves eigenvalues,
it then follows that it is enough to show that $\widetilde{A}^T$ is a
stable matrix.  If we can show that $\lim_{k \rightarrow \infty}
\K^{k}(V) = 0$, then by \cite[Proposition 2.5]{CostaBook} it follows
that $\widetilde{A}^T$ is a stable matrix.  In order to show $\lim_{k
  \rightarrow \infty} \K^{k}(V) = 0$, we use a Lyapunov-like argument.
That is, we define $W(V)=\mathrm{tr}(V^TP)$ and we show that it
behaves like a Lyapunov function. More precisely, denote by
$\mathcal{P}$ the set of all $n \times n$ positive semi-definite
matrices.  Notice that $\K:\mathcal{P} \rightarrow \mathcal{P}$ is a
continuous map, if $\mathcal{P}$ is viewed as a metric space with the
metric $d(V_1,V_2)=||V_1-V_2||_{F}=\mathrm{tr}((V_1-V_2)^T(V_1-V_2))$.  
Notice moreover that $\K(0)=0$.
Hence, if $W$ satisfies the properties below, then by \cite[Theorem
2.12]{BhatiaSzego} the dynamical system \(V_{k+1}=\K(V_k) \) defined
on $\mathcal{P}$ is globally asymptotically stable for the equilibrium
point $0$, i.e. $\lim_{k \rightarrow \infty} \K^k(V) = 0$ for any $V
\in \mathcal{P}$.  The properties $W$ has to satisfy are the
following.
\begin{enumerate}
\item $W:\mathcal{P} \rightarrow \mathbb{R}$ is continuous, $W(S) \ge
0$ for any $S \in \mathcal{P}$,
\item $W(S)=0$ iff $S=0$, for all $S \in \mathcal{P}$,
\item $W(\K(S)) < W(S)$ for all $S \in \mathcal{P}$, $S \ne 0$,
\item $W$ is radially unbounded\footnote{Using the terminology of
  \cite{BhatiaSzego}, this property implies that $W$ is uniformly
  unbounded}, more specifically
\[ \lim_{||S||_F \rightarrow \infty} W(S)=+\infty. \]
\end{enumerate}
The first two properties follow from the definition.  To see \(
W(\K(V)) \le W(V) \mbox{ for any } V \ge 0 \), notice that
$V^{\frac{1}{2}}$ exists \footnote{here $V^{\frac{1}{2}}$ is the unique matrix
  such that $V=(V^{\frac{1}{2}})^2$)} 
$(V^{\frac{1}{2}})^T=V^{\frac{1}{2}}$ and that
\[
\begin{split}
W(V)&=\mathrm{tr}(V^TP)=\mathrm{tr}(VP)=\mathrm{tr}(V^{\frac{1}{2}}PV^{\frac{1}{2}})\\
&= \sum_{i=1}^{n} e^T_iV^{\frac{1}{2}}PV^{\frac{1}{2}}e_i \\
&\ge \sum_{i=1}^{n} \sum_{q \in Q} e^T_iV^{\frac{1}{2}}A^T_{q}PA_{q}V^{\frac{1}{2}}e_i \\
&= \sum_{q \in Q} \mathrm{tr}((V^{\frac{1}{2}}A^T_{q})PA_{q}V^{\frac{1}{2}}) \\
&= \sum_{q \in Q} \mathrm{tr} (A_{q}VA^T_{q}P)=\sum_{q \in Q} \mathrm{tr}((A_qVA_q^T)^TP)=W(\K(V)).
\end{split}
\]
In order to see that $W$ is radially unbounded, notice that there
exists $m > 0$ such that $P-mI > 0$. Hence for any $S \ge 0$,
\[
\begin{split}
W(S)&=\mathrm{tr}(S^TP)=\mathrm{tr}(SP)=\mathrm{tr}(S^{\frac{1}{2}}PS^{\frac{1}{2}})\\
&= \sum_{i=1}^{n} e^T_iS^{\frac{1}{2}}PS^{\frac{1}{2}}e_i \\
&\ge m \sum_{i=1}^{n} e_i^TSe_i =
\mathrm{tr}(V)=||S^{\frac{1}{2}}||^2_{F}
\end{split}
\]
Since the Frobenius norm is subadditive, it follows that $||S||_F \le
||S^{\frac{1}{2}}||^2$ and hence
\begin{equation}
\label{nice_stab:proof:eq2}
\forall S \ge 0:  m||S||_{F} \le W(S).
\end{equation}
Hence, $W(S)$ is radially unbounded, i.e., $\lim_{||S|| \rightarrow
  \infty} W(S)=+\infty$.
\end{proof}

\section{}\label{appl8}

\begin{proof}[Proof of Lemma \ref{nice_stab:min}]
Notice that state-space isomorphism preserves strong
stability. Indeed, if $\Sigma_1$ and $\Sigma_2$ are related by an
isomorphism $\MORPH$ and their corresponding matrices are $A_q$ and
$F_q$, $q \in Q$, then $F_q=\MORPH A_q \MORPH^{-1}$ and hence,
$\sum_{q \in Q} F^T_q \otimes F_q^T = (\MORPH^{-T} \otimes
\MORPH^{-T}) (\sum_{q \in Q} A_q^T \otimes A^T_q)(\MORPH^T \otimes
\MORPH^T)$. That is, $\sum_{q \in Q} F_q^T \otimes F_q^T$ and $\sum_{q
  \in Q} A_q^T \otimes A_q^T$ are similar matrices, and hence they
have the same eigenvalues.

Since all equivalent minimal realizations are isomorphic, it then
follows that it is enough to show that if $\Sigma$ is strongly stable,
then the result of application of Procedure \ref{LSSmin} is also
strongly stable.  Note that Procedure \ref{LSSmin} is just the
subsequent application of Procedure \ref{LSSreach} or Procedure
\ref{LSSobs}, hence it is enough to show that Procedures
\ref{LSSreach} -- \ref{LSSobs} preserve strong stability.  Recall from
Remark \ref{rem:dual} the notion of the dual system
$\Sigma^T$ and recall that if $\hat{\Sigma}$ is the result of applying
Procedure \ref{LSSreach} to $\Sigma^T$, then $\hat{\Sigma}^T$ is the
result of applying Procedure \ref{LSSobs} to $\Sigma$. Notice that the
matrices of $\Sigma^T$ are $A_q^T$, $q \in Q$, and $(\sum_{q \in Q}
A_q \otimes A_q)^T=\sum_{q \in Q} A_q^T \otimes A_q^T$. Since taking
transposes preserves eigenvalues, it then follows that an \BLSS is
stable if and only if its dual system is stable.  Hence, if we show
that Procedure \ref{LSSreach} preserves strong stability, then by a
duality argument we get that Procedure \ref{LSSobs} also preserves
strong stability.

Thus, it is left to show that Procedure \ref{LSSreach} preserves strong
stability.  If $\Sigma$ is strongly stable, then by Lemma
\ref{nice_stab:proof}, there exists $P > 0$ such that $P-\sum_{q \in Q}
A_q^TPA_q > 0$.  Consider the partitioning of the matrix $A_q$
described in Procedure \ref{LSSreach} to $\Sigma$, and let
$P=\begin{bmatrix} P_{11} & P_{12} \\ P_{21} & P_{22} \end{bmatrix}$
be a compatible partitioning of $P$, i.e., $P_{11}$ is $r \times r$.
From \textbf{Observation 1} in the proof of Lemma
\ref{inv:theo:lemma2} it then follows that
\[ P-\sum_{q \in Q} A_q^TPA_q =
\begin{bmatrix} P_{11}-\sum_{q \in Q} (A_q^R)^TP_{11}A_{q}^R & \star \\
\star & \star
\end{bmatrix}.
\]
Hence, $P-\sum_{q \in Q} A_q^TPA_q > 0$ implies that $P_{11}-\sum_{q
  \in Q} (A_{q}^R)^TP_{11}A_{q}^R > 0$. From $P > 0$ it follows that
$P_{11} > 0$. From Lemma \ref{nice_stab:proof}, it follows that
$\sum_{q \in Q} (A_{q}^R)^T \otimes (A_{q}^R)^T$ is stable, i.e., the
result of applying Procedure \ref{LSSreach} to $\Sigma$ is strongly
stable.
\end{proof}

\section{}\label{appl9}

\begin{proof}[Proof of Lemma \ref{niceGramLemma0}]
That nice observability grammian exists, it is unique and it is
positive semi-definite follows from Lemma \ref{nice_stab:proof} by
setting $\mathcal{G}=\sum_{q \in Q} C_q^TC_q$ and $F_q=A_q$.  The corresponding
statement for observability grammians follows from Lemma
\ref{nice_stab:proof}, by applying it to $F_q=A_q^T$, $q \in Q$ and
$\mathcal{G}=\sum_{q \in Q} B_qB_q^T$ and by noticing that $(A_q^T\otimes
A_q^T)^T=A_q \otimes A_q$; hence, $\sum_{q \in Q} A_q^T \otimes
A_q^T$ is stable if and only if $\sum_{q \in Q} A_q \otimes A_q$ is
stable.
\end{proof}

\section{}\label{appsandberg}

\begin{proof}[Proof of Lemma \ref{BalanceClaim8.1}]
The proof is inspired by the PhD thesis \cite{SandbergThesis}.
Without loss of generality, we assume that $\Sigma$ is already
balanced and hence $\Sigma_{\mathrm{bal}}=\Sigma$.  Assume that the
balanced observability and controllability grammians are of the
following form.
\[
\Lambda=\begin{bmatrix} \Lambda_1 & 0 \\ 0 & \beta
\end{bmatrix}.
\]
We use the notation of the partitioning in
\eqref{BalancedTruncate:eq1}. Moreover, we will use the continuous
time notation generically, e.g., $u(t)$ will denote either a
continuous time input or a discrete time input depending on the
context. 

Fix an input $u \in \mathcal{U}$ and a switching signal $q \in \mathcal{Q}$ 
and denote by $x(t)$ the
corresponding state trajectory of $\Sigma$ and by $\hat{x}(t)$ the
corresponding state trajectory of the reduced order model
$\hat{\Sigma}$. Consider the decomposition $x(t)=(x_1(t),x_2(t))$
where $x_1(t) \in \mathbb{R}^{n-1}$, and define
\[ z(t)=A_{q(t),21}\hat{x}(t)+B_{q(t),2}u(t). \] With this notation,
consider the following vectors
\[ X_c(t)=\begin{bmatrix} x_1(t)+\hat{x}(t) \\ x_2(t)
\end{bmatrix}
\mbox{, \ \ }
X_o(t)=\begin{bmatrix} x_1(t)-\hat{x}(t) \\
x_2(t)
\end{bmatrix}.
\]
An easy calculation reveals that 
(with $\delta$ the derivative- or forward shift operator)
\[
\begin{split}
& \delta X_c(t) = A_{q(t)}X_{c}(t)-\begin{bmatrix} 0 \\ z(t) \end{bmatrix} + 2B_{q(t)}u(t) \\
& \delta X_o(t) = A_{q(t)}X_{o}(t)+\begin{bmatrix} 0 \\
z(t) \end{bmatrix}
\end{split}.
\]
We will show that

\begin{Lemma}\label{l1}
\begin{equation}
\begin{split}\label{ClaimBalance8.1:6B}
4\beta^{2} &||u||_{2}^{2}\\
&\ge 
\begin{cases} \int_{0}^{t} X_o(s)^TC_{q(s)}^TC_{q(s)}X_o(s)ds &\mbox{\ \ (cont.)} \\
\sum_{s=0}^{t-1} X_o(s)^TC_{q(s)}^TC_{q(s)}X_o(s) &\mbox{\ \ (disc.)} 
\end{cases}
\end{split}
\end{equation}
\end{Lemma}
Before proving Lemma~\ref{l1}, notice that \( C_{q(t)}X_o(t)=y(t)-\hat{y}(t) \), where $y(t)$ is
the output trajectory of $\Sigma$ and $\hat{y}(t)$ is the output
trajectory of $\hat{\Sigma}$.  Hence, \eqref{ClaimBalance8.1:6B} is
equivalent to
\[ ||y-\hat{y}||_{2}^{2} \le 4\beta^{2}||u||_{2}^{2} \]
From this Lemma~\ref{BalanceClaim8.1} follows.

In order to prove Lemma~\ref{l1}, we proceed as follows. Notice that
\begin{subequations}\label{ClaimBalance8.1:2B}
\begin{align*}
\delta(&X_o^T(t)\Lambda X_{o}(t))\\
&= 
\begin{cases}
2X_o(t)^TA^T_{q(t)}\Lambda X_o(t) +2\beta z(t)^Tx_2(t) &\mbox{\ \  (cont.) } \\
\\
X_o^T(t)A_{q(t)}^T \Lambda A_{q(t)}X_o(t)\\
\quad + 2\beta X_o(t)^TA^T_{q(t)}\begin{bmatrix} 0 \\
z(t) \end{bmatrix} + \beta ||z(t)||_2^2 &\mbox{\ \ (disc.)}.
\end{cases}
\end{align*}
\end{subequations}
In the derivation above we used that
\[ \Lambda\begin{bmatrix} 0 \\ z(t) \end{bmatrix} =
\beta \begin{bmatrix} 0 \\ z(t) \end{bmatrix}.
\]
Using this and the fact that $\Lambda$ satisfies the observability
grammian inequality \eqref{ContObsGram}, it follows that
\begin{align}
\delta(&X_o^T(t) \Lambda X_{o}(t)) \nonumber\\
&\le \begin{cases}
- X_o(t)^TC_{q(t)}^TC_{q(t)}X_o(t) + 2\beta z(t)^Tx_2(t) &\mbox{\ \ \
  (cont.) } \\
\\
- X_o(t)^TC_{q(t)}^TC_{q(t)}X_o(t) + X_o(t)^T\Lambda X_o(t)\\
\quad + 2\beta X_o(t)^TA^T_{q(t)}\begin{bmatrix} 0 \\ z(t) \end{bmatrix} +
\beta ||z(t)||_2^2 &\mbox{\ \ (disc.)}.
\end{cases}
\label{ClaimBalance9.1}
\end{align}
By noticing that $X_o(0)=0$ and hence
\begin{align*}
0 &\le X^T_o(t)\Lambda X_o(t)\\
&= 
\begin{cases}
\int_{0}^{t} \frac{d}{dr}(X_o^T(r) \Lambda X_{o}(r))|_{r=s}ds &\mbox{\
  \ (cont.) } \\
\\
\sum_{s=0}^{t-1}\Big(X_o^T(s+1) \Lambda X_{o}(s+1)\\ 
\qquad- X_o^T(s) \Lambda
X_{o}(s)\Big), &\mbox{\ \ (disc.)}
\end{cases}
\end{align*}
and combing it with \eqref{ClaimBalance9.1}, it follows that
\begin{subequations}\label{ClaimBalance9.2}
\begin{align*}
0 &\le - \int_{0}^{t} X_o(s)^TC_{q(s)}^TC_{q(s)}X_o(s)ds&\\
&\quad + 2\beta \int_{0}^{t} z(s)^Tx_2(s)ds, &\mbox{\ \ \ (cont.)}\\
& &\\
0 &\le -\sum_{s=0}^{t-1} X_o(s)^TC_{q(s)}^TC_{q(s)}X_o(s)&\\
&\quad + \sum_{s=0}^{t-1} (2 \beta X_o(s)^TA^T_{q(s)}\begin{bmatrix} 0 \\
z(s) \end{bmatrix} + \beta ||z(s)||_2^2). &\mbox{\ \ \
  (disc.)} 
\end{align*}
\end{subequations}
If we can show that
\begin{Lemma}
\label{l2}
\begin{align*}
4\beta&||u||_{2}^{2}\\
&\ge \begin{cases}
2\int_{0}^{t}z(s)^Tx_2(s)ds  &\mbox{\ \ \ (cont.)} \\
\sum_{s=0}^{t-1} (2X_o(s)^TA^T_{q(s)}\begin{bmatrix} 0 \\
z(s) \end{bmatrix} + ||z(s)||_2^2) &\mbox{\ \ \ (disc.)}
\end{cases}
\end{align*}
\end{Lemma}
then Lemma~\ref{l1} follows.

\begin{proof}[Proof of Lemma \ref{l2}]
We split the proof of Lemma \ref{l2} into two parts: one for the
continuous-time case, and one for the discrete-time case.

{\bf Continuous-time}

Notice that by applying Lemma \ref{Claim0} with $A=A_q^T$, $S=B_q^T$
and using the fact that $\Lambda$ is a controllability grammian, it
follows that
\[
\begin{bmatrix}
A^T_{q}\Lambda^{-1}+\Lambda^{-1}A_q & \Lambda^{-1}B_q  \\
B_q^T\Lambda^{-1} & -I
\end{bmatrix} \le 0.
\]
Hence,
\[
\begin{split}
\frac{d}{dt}& (X_c(t)^T\Lambda^{-1} X_c(t))\\
&= \begin{bmatrix} X_c(t) \\ 2u(t)
\end{bmatrix}^T
\begin{bmatrix}
A^T_{q(t)}\Lambda^{-1}+\Lambda^{-1}A_{q(t)} & \Lambda^{-1}B_{q(t)} \\
B_{q(t)}^T\Lambda^{-1} & -I
\end{bmatrix}
\begin{bmatrix}
X_c(t) \\ 2u(t)
\end{bmatrix}\\
&\quad + 4||u(t)||_{2}^{2}-2X^T_c(t)\Lambda^{-1}\begin{bmatrix} 0 \\ z(t) \end{bmatrix}\\
&\leq 4||u(t)||_{2}^{2}-2X^T_c(t)\Lambda^{-1}\begin{bmatrix} 0 \\
z(t) \end{bmatrix}
\end{split}
\]
A simple computation reveals that $X^T_c(t)\Lambda^{-1}\begin{bmatrix}
0 \\ z(t) \end{bmatrix}=\frac{1}{\beta} x^T_2(t)z(t)$, and thus
\[ \frac{d}{dt} (X^T_c(t)\Lambda^{-1}X_c(t)) \le
4||u(t)||_{2}^{2}-2\frac{1}{\beta}x_2^T(t)z(t)
\]
By noticing that $X_c(0)=0$, we get that
\[
\begin{split}
X_c(t)^T\Lambda^{-1}X_c(t) &=
\int_{0}^{t} \frac{d}{ds} (X_c(s)^T\Lambda^{-1}X_c(s))ds \\
&\leq 4\int_{0}^{t} ||u(s)||^2ds-2 \int_{0}^{t} \frac{1}\beta
x_{2}^T(s)z(s)ds
\end{split}
\]
Since $X_c(t)^T\Lambda^{-1}X_c(t) \ge 0$,
\begin{equation}
\label{ClaimBalance8.1:5CC}
2\int_{0}^{t} z(s)^Tx_2(s)ds
\le 4\beta \int_0^{t} ||u(s)||_{2}^{2}ds \le 4\beta ||u||_{2}^{2}
\end{equation}

\textbf{Discrete-time}

By applying Lemma \ref{Claim0} with $A=A_q^T$ and $S=B_q^T$ for the
discrete-time case, equation~\eqref{ContContrGram} for $\CP=\Lambda$
can be rewritten as
\begin{align}
&\begin{bmatrix} A_q^T\Lambda^{-1}A_q-\Lambda^{-1} & A_q^T\Lambda^{-1}B_q \\
B_q^T\Lambda^{-1}A_q & B^T_q\Lambda^{-1}B_{q} -I
\end{bmatrix} \le 0.\label{ClaimBalance8.1:3DD}
\end{align}
Hence, by using \eqref{ClaimBalance8.1:3DD},
\begin{equation}
\label{ClaimBalance8.1:41DD}
\begin{split}
X_c  &(t+1)^T \Lambda^{-1}X_c(t+1) - X^T_c(t)\Lambda^{-1}X_c(t)\\
&=\begin{bmatrix} X_c(t) \\ 2u(t) \end{bmatrix}^T
\begin{bmatrix} A_q^T\Lambda^{-1}A_q-\Lambda^{-1} & A_q^T\Lambda^{-1}B_q \\
B_q^T\Lambda^{-1}A_q & B^T_q\Lambda^{-1}B_{q} -I
\end{bmatrix}
\begin{bmatrix} X_c(t) \\ 2u(t) \end{bmatrix}\\
&\quad - 2(A_{q(t)}X_c(t)+2B_{q(t)}u(t))^T\Lambda^{-1}\begin{bmatrix} 0 \\
z(t)  \end{bmatrix}\\ 
&\quad+ z^T(t)\Lambda^{-1}z(t) + 4||u(t)||^{2}_2\\ 
&\le
-2(A_{q(t)}X_c(t)+2B_{q(t)}u(t))^T\Lambda^{-1}\begin{bmatrix} 0 \\ z(t)  \end{bmatrix}\\
&\quad + z^T(t)\Lambda^{-1}z(t)+4||u(t)||^{2}_2
\end{split}
\end{equation}
Notice that
\[
\begin{split}
\begin{bmatrix} 2\hat{x}(t+1) \\ 0 \end{bmatrix} &= X_{c}(t+1)-X_{o}(t+1)\\
& =A_{q(t)}X_c(t)+2B_{q(t)}u(t)\\
&\quad -A_{q(t)}X_o(t)-2\begin{bmatrix} 0 \\
z(t) \end{bmatrix},
\end{split}
\]
from which it follows that
\begin{equation}
\label{ClaimBalance8.1:42DD}
A_{q(t)}X_c(t)+2B_{q(t)}u(t)=A_{q(t)}X_o(t)+2\begin{bmatrix} \hat{x}(t+1) \\ z(t) \end{bmatrix}.
\end{equation}
By substituting \eqref{ClaimBalance8.1:42DD} into
\eqref{ClaimBalance8.1:41DD} and using that
\[
z^T(t)\Lambda^{-1}z(t)=\frac{1}{\beta} ||z(t)||^2_2
\]
and 
\begin{align*}
(A_{q(t)}X_c(t)+&2B_{q(t)}u(t))^T\Lambda^{-1}\begin{bmatrix} 0 \\
z(t) \end{bmatrix}\\
&=\frac{1}{\beta}
(A_{q(t)}X_c(t)+2B_{q(t)}u(t))^T\begin{bmatrix} 0 \\
z(t) \end{bmatrix}, 
\end{align*}
it follows that
\begin{equation}
\label{ClaimBalance8.1:43DD}
\begin{split}
X_c(t&+1)^T\Lambda^{-1}X_c(t+1) - X^T_c(t)\Lambda^{-1}X_c(t)\\
&\leq -2\frac{1}{\beta}X_o(t)^TA_{q(t)}^T\begin{bmatrix} 0 \\
z(t) \end{bmatrix} + 4||u(t)||_{2}^{2}-\frac{3}{\beta}||z(t)||^{2}\\
&\leq  -2\frac{1}{\beta}X_o(t)^TA^T_{q(t)}\begin{bmatrix} 0 \\
z(t) \end{bmatrix} - \frac{1}{\beta}||z(t)||^{2} + 4||u(t)||_{2}^{2}.
\end{split}
\end{equation}
Since $X_c(t)^T\Lambda^{-1}X_c(t) \ge 0$ and $X_c(0)=0$,
\eqref{ClaimBalance8.1:43DD} can be rewritten as
\begin{equation*}
\begin{split}
0 &\le X_c(t)^T\Lambda^{-1}X_c(t)\\
& =\sum_{s=0}^{t-1} (X_c(s+1)^T\Lambda^{-1}X_c(s+1)-X_c(s)^T\Lambda^{-1}X_c(s)) \\
& \le \frac{1}{\beta} (\sum_{s=0}^{t-1} -2X_o(s)^TA^T_{q(s)}\begin{bmatrix}
0 \\ z(s) \end{bmatrix} -||z(s)||_{2}^{2})+4||u||_{2}^{2}.
\end{split}
\end{equation*}
From which it follows that
\begin{equation}
\label{ClaimBalance8.1:5DD}
\sum_{s=0}^{t-1} (2X_o(s)^TA_{q(s)}\begin{bmatrix} 0 \\ z(s) \end{bmatrix} +||z(s)||_{2}^{2}) 
\le 4\beta ||u||_{2}^{2}
\end{equation}
\end{proof}
\end{proof}

\section{}\label{appl10}
\begin{proof}[Proof of Lemma \ref{beck:theo}]
  In order to present the proof, we will use the following auxiliary result.
  Using the terminology of \cite{BeckThesis}, define
  \[
    \begin{split}
     & A=
     \begin{bmatrix} 0 & A_1 & \ldots & A_{\QNUM}  \\
     I & 0 & \ldots & 0 \\
    \vdots & \vdots & \ldots  &  \vdots  \\ \
     I & 0 & \ldots & 0 \\
    \end{bmatrix} \mbox{,\ \ }
 B=\begin{bmatrix}
 B_1 & \ldots & B_{\QNUM}\\
 0 & \ldots & 0 \\
 \vdots & \ldots & \vdots \\
 0 & \ldots & 0  \\
 \end{bmatrix} \\
 & C=\begin{bmatrix}
0 & C_1 & \ldots & C_{\QNUM}
\end{bmatrix}.
\end{split}
\]
It then follows that 
\( M_{\Sigma}=  
\left[
\begin{array}{c|c}
A &B \\ \hline
C &0
\end{array}
\right]
\). 
Denote by $\mathcal{T}$ the set of all matrices of the form
$\mathrm{diag}(S_1,\ldots,S_{\QNUM+1})$ such that $S_{i}$ are $n \times n$ matrices.
Denote by $\mathcal{T}^{+}$ the subset of all the matrices 
$\mathrm{diag}(S_1,\ldots,S_{\QNUM+1})$ such that $S_{i}$ are $n \times n$ positive
definite matrices. In order to define grammians and quadratic stability,
in \cite{BeckThesis} expression of the following form were investigated
\( A^{T}XA - X \) and \( AYA^T-Y \) for $X,Y \in \mathcal{T}^{+}$.
An easy calculation leads to the following proposition.
\begin{Proposition}
\label{beck:theo:prop1}
 Assume that $X=\mathrm{diag}(S_1,\ldots,S_{\QNUM+1}) \in \mathcal{T}^{+}$.
 The $i$th $n \times n$ diagonal block of $A^TXA-X$ is of the form
 \( A^T_{i}S_1A_{i}-S_{i} \) for $i=2,\ldots,\QNUM$ and
 it is of the form \(\sum_{i=2}^{\QNUM+1} S_i - S_1\) for $i=1$.
 Similarly, the first $n \times n$ diagonal block of 
 \( AXA^T-X \) equals $\sum_{q \in Q} A_qX_{q+1}A_q^T-X_1$ and the
  $i$th diagonal $n \times n$ block equals $P_1$ for $i=2,\ldots,\QNUM+1$.
\end{Proposition}
Based on Proposition \ref{beck:theo:prop1} we can show the following.

\textbf{$M_{\Sigma}$ is quadratically stable $\implies$ $\Sigma$ strongly stable}
  If $M_{\Sigma}$ is quadratically stable, then for some $X \in \mathcal{T}^{+}$,
  $A^{T}XA-X < 0$. By Proposition \ref{beck:theo:prop1} the diagonal elements of
  $A^TXA-X$ are $\sum_{q=2}^{\QNUM} S_{q+1} - S_1$ and
  $A^T_qS_1A_q - S_{q+1}$, $q \in Q$. If $A^TXA-X$ is negative definite, then so are its
  diagonal elements and hence $\sum_{q=2}^{\QNUM} S_{q+1} - S_1 < 0$ and
  $A^T_qS_1A_q - S_{q+1} < 0$, $q \in Q$. Hence,
  $\sum_{q \in Q} A_q^TS_1A_q - \sum_{q \in Q} S_{q+1} < 0$. Using
  $S_1 > \sum_{q \in Q} S_{q+1}$, it follows that
  $\sum_{q \in Q} A_q^TS_1A_q - S_1 < 0$. Then from Lemma \ref{nice_stab:proof} it follows
  that $\Sigma$ is strongly stable.

\textbf{If $\CP$ and $\Q$ are the controllability and observability grammians of
        $M_{\Sigma}$, then $\CP$ and $\Q$ are controllability and
        observability grammians of $\Sigma$ and \eqref{beck:theo:eq1} holds.}

  If $\CP$ and $\Q$ (more precisely, $P_1$ and $Q_1$) 
  satisfy \eqref{beck:theo:eq1}, then they are clearly
  controllability and observability grammians of $\Sigma$.
  Hence, it is enough to show that \eqref{beck:theo:eq1} holds. To this end, recall
  that if 
  $\CP$ and $\Q$ are controllability (resp. observability) grammians of $M_{\Sigma}$,
  then
  \( A^{T}\Q A + C^TC-\Q \le 0 \) and \( A\CP A^T + BB^T - \CP \le 0 \).
  From Proposition \ref{beck:theo:prop1} it then follows that
  the diagonal $n \times n$ blocks of $A^T\Q A+C^TC-\Q$
  are of the form 
  $\sum_{q \in Q} Q_{q+1} - Q_1$, $A^T_{q}Q_1A_q +C_qC_q^T- Q_{q+1}$.
 Hence, $\sum_{q \in Q} Q_{q+1} - Q_1 \le 0$, $A^T_{q}Q_1A_q +C^T_qC_q- Q_{q+1} \le 0$.
 By taking the sums of $A_{q}^TQ_1A_{q}-Q_{q+1}$ and 
 by taking into account $\sum_{q \in Q} Q_{q+1} - Q_1 \le 0$, it follows that
 \( \sum_{q \in Q} A_q^TQ_1 A_q+C_q^TC_q - Q_1  \le 0\).
  From Proposition \ref{beck:theo:prop1} it then follows that
  the first $n \times n$ diagonal block of $A\CP A^T + BB^T - \CP$ equals 
  $\sum_{q \in Q} (A_qP_{q+1}A_q^T+B_qB^T_q) - P_1$,  and all the other diagonal
  blocks are $P_1-P_{q+1}$, $q \in Q$. Hence, $P_1 \le P_{q+1}$ and 
  $\sum_{q \in Q} (A_qP_{q+1}A_q^T+B_qB^T_q) - P_1 \le 0$. Since
  then $A_{q}P_1A_q^T < A_{q}P_{q+1}A_{q}^T$, it follows that
  \( \sum_{q \in Q} (A_qP_{1}A_q^T+B_qB^T_q) - P_1 \le 0 \) holds.
   
  \textbf{$M_{\Sigma}$ is minimal $\iff$ $\Sigma$ is minimal} By
  \cite{BeckThesis}, $M_{\Sigma}$ is minimal if and only if it is
  reachable and observable.  The latter conditions are equivalent to
  $\mathbb{R}^{n}=\sum_{k=0}^{\infty}
  \sum_{i_1,\ldots,i_k=1}^{\QNUM+1} \IM A_{i,i_{k}}\cdots
  A_{i_2,i_1}G_{i_1}$ and $\{0\}=\bigcap_{k=0}^{\infty}
  \bigcap_{i_1,\ldots,i_k=1}^{\QNUM+1}$ $ \ker
  C_{i_k}A_{i_{k},i_{k-1}}\cdots A_{i_1,i}$ for all $i=1,\ldots,
  \QNUM$.  Here $A_{1,q+1}=A_q$, $A_{q+1,1}=I_n$, $q \in Q$, and
  $A_{i,j}=0$ otherwise, Similarly, $C_1=0$ and $C_{q}=C_{q-1}$ for $q
  > 1$.  Finally $B=\begin{bmatrix} G^T_1 & \ldots &
  G^T_{\QNUM+1} \end{bmatrix}^T$ and thus $G_1=\begin{bmatrix} B_1 &
  \ldots & B_{\QNUM} \end{bmatrix}$ and $G_{q}=0$ for $q > 1$.  It
  then follows that $A_{i,i_{k}}\cdots A_{i_2,i_1}G_{i_1}=A_{v}G_1$,
  if $v=q_1\cdots q_l$ and $i_1\cdots i_k=1(q_1+1)1(q_2+1)\cdots
  1(q_l+1)1$ or $i_1\cdots i_k=1(q_1+1)1(q_2+1)\cdots 1(q_l+1)1(q+1)$ for some $q \in Q$, and
  $A_{i,i_{k}}\cdots A_{i_2,i_1}G_{i_1}=0$ otherwise.  Similarly,
  $C_{i_k}A_{i_{k},i_{k-1}}\cdots A_{i_1,i}=C_{q_l}A_{v}$, if
  $vq_l=q_1\cdots q_l$, $i_1\cdots i_k=(q_1+1)1(q_2+1) \cdots (q_{l-1}+1)1(q_l+1)$ or
  and it is zero otherwise.
  Hence, by Remark \ref{lin:rem1.5} reachability of $M_{\Sigma}$ is
  equivalent to span-reachability of $\Sigma$ and observability of
  $M_{\Sigma}$ is equivalent to observability of $\Sigma$.
  
\textbf{$M_{\Sigma}$ is balanced $\implies$ $\Sigma$ is balanced}
  Assume that $\CP=\Q$ diagonal, then the first $n \times n$ block of 
  $\CP=\Q$ is also diagonal and it is an observability and reachability grammian of
  $\Sigma$.  That is, $\Sigma$ is balanced.

\textbf{$M_{\hat{\Sigma}}$ arises from balanced truncation}
  Assume that  $\CP=\Q=\mathrm{diag}(\Lambda_1,..$ $..,\Lambda_{\QNUM+1})$ and
  assume that in Procedure \ref{BalancedTruncate} we discard the 
  $n-r$ smallest singular values of $\Lambda_1$.  Let us apply balanced truncation to 
  $M_{\Sigma}$ by discarding the $n-r$ smallest singular values from
  $\Lambda_1,\ldots,\Lambda_{\QNUM+1}$.
  From the formula presented in \cite{BeckThesis} it then follows that
  the resulting uncertain system equals $M_{\hat{\Sigma}}$. 
\end{proof}

\section{}\label{appl11}

\begin{proof}[Proof of Lemma \ref{kotsalis:theo}]
 \textbf{Proof of Part \ref{kotsalis:theo:part1}}
  Notice that $\sum_{q \in Q} A_q^TPA_q - P < 0$ is equivalent to
  $\sum_{q \in Q} p ((\frac{1}{p} A^T_qPA_q) - P < 0$. The existence of a
  positive definite solution to former LMI is 
  equivalent strong stability of $\Sigma$, and the latter LMI is equivalent to 
  mean-square stability of $\Sigma_{\mathrm{st}}$.


 \textbf{Proof of Part \ref{kotsalis:theo:part2}}
   Notice that  
  \begin{align*} \sum_{q \in Q} (A_{q}\CP A^T_q&+B_qB_q^T)\\
      &=\sum_{q \in Q} p((\frac{1}{\sqrt{p}}A_{q})\CP (\frac{1}{\sqrt{p}}A^T_q)+
     (\frac{1}{\sqrt{p}}B_q)(\frac{1}{\sqrt{p}}B_q)^T)
  \end{align*} and 
  \begin{align*} \sum_{q \in Q} (A^T_{q}\Q A_q&+C^T_qC_q)\\
  &= \sum_{q \in Q} p((\frac{1}{\sqrt{p}}A^T_{q})\Q
  (\frac{1}{\sqrt{p}}A_q)+
  (\frac{1}{\sqrt{p}}C^T_q)(\frac{1}{\sqrt{p}}C_q)).
  \end{align*}
  From this it the first part of the claim follows.
  Since clearly
  $A_q\CP A_q^T+B_qB_q^T - \CP \le \sum_{q \in Q} (A_{q}\CP A^T_q+B_qB_q^T) -\CP$ and
  $A^T_q\Q A_q+C^T_qC_q - \Q \le \sum_{q \in Q} (A^T_{q}\Q A_q+C^T_qC_q) -\Q$, it follows
  that if $\CP$ and $\Q$ satisfy \eqref{kotsalis:theo:eq1}, then 
  they are controllability resp. observability of $\Sigma$.

  As for the second part of the claim, if $\CP$ and $\Q$ are nice controllability
  resp. observability grammians, then 
  they satisfy  \eqref{kotsalis:theo:eq1} and hence they are also controllability
  resp. observability grammians of $\Sigma_{\mathrm{st}}$.

\textbf{Proof of Part \ref{kotsalis:theo:part3}} 
  Assume that we apply Procedure \ref{BalancedTruncate} to grammians $\CP$ and $\Q$
  which satisfy \eqref{kotsalis:theo:eq1}. Let $\MORPH$ be the state-space isomorphism
  which renders $\Sigma$ balanced. Denote the resulting balanced \LSS\ by 
  $\bar{\Sigma}$.
   Since $\Sigma_{\mathrm{st}}$ has the same state-space
  as $\Sigma$, we can apply the state-space transformation to obtain a 
  jump-linear system
  \[
  \bar{\Sigma}_{\mathrm{st}}=\left\{
  \begin{split}
   & z(t+1)=\bar{A}_{\theta(t)}z(t)+\bar{B}_{\theta(t)}u(t) \\
   & \widetilde{y}(t)=\bar{C}_{\theta(t)}z(t)
   \end{split}\right.
  \]
  where $\bar{A}_q=\frac{1}{\sqrt{p}}\MORPH A_q \MORPH^{-1}$,  
  $\bar{B}_q=\frac{1}{\sqrt{p}} \MORPH B_q$, $\bar{C}_q=\frac{1}{\sqrt{p}}C_q\MORPH^{-1}$. It is then easy to see
  that the matrices $(\MORPH^{-1})^T\Q\MORPH^{-1}=\MORPH \CP \MORPH^{T}=\Lambda$ are 
  equal and diagonal and satisfy \eqref{kotsalis:theo:eq1} with
  $A_q$, $B_q$ and $C_q$ being replaced by
  $\MORPH A_q \MORPH^{-1}, \MORPH B_q$ and $C_q \MORPH^{-1}$, $q \in Q$ respectively.
  Hence they are 
  also grammians of $\bar{\Sigma}_{\mathrm{st}}$, i.e. $\bar{\Sigma}_{\mathrm{st}}$ is
  balanced according to \cite{Kotsalis1}. Finally, if
  $\Lambda=\mathrm{diag}(\sigma_1,\ldots,\sigma_n)$, $\sigma_1 \ge \cdots \sigma_n$
  and we truncate the singular values $\sigma_{r+1},\ldots,\sigma_n$, then
  Procedure \ref{BalancedTruncate} returns the system
  $\hat{\Sigma}=(n,Q, \{\hat{A}_q,\hat{B}_q,\hat{C}_q\}_{q \in Q})$, where
  $\hat{A}_{q}$ the is upper left $r \times r$ block of $\MORPH A_q \MORPH^{-1}$,
  $\hat{B}_q$ is formed by the first $r$ rows of $\MORPH B_q$, and
  $\hat{C}_q$ is formed by the first $r$ columns of $C_q\MORPH^{-1}$.
  But then the stochastic system 
  \[
  \hat{\Sigma}_{\mathrm{st}}=\left\{
  \begin{split}
   & \bar{z}(t+1)=\frac{1}{\sqrt{p}}\hat{A}_{\theta(t)}\bar{z}(t)+\frac{1}{\sqrt{p}}\bar{B}_{\theta(t)}u(t) \\
   & \widetilde{y}(t)=\frac{1}{\sqrt{p}}\bar{C}_{\theta(t)}\bar{z}(t)
   \end{split}\right.
  \]
  is easily seen to coincide with the result of applying balanced truncation to 
  $\bar{\Sigma}_{\mathrm{st}}$.

\textbf{Proof of Part \ref{kotsalis:theo:part4}} 
  With a slight abuse of notation, for $u \in \mathcal{U}$, $v \in Q^{+}$,
  we will denote by $Y(u,v)$ the value
  $Y^{\Sigma}(u,q)(t)$, where $t=|v|-1$ and the switching signal $q$ such that
  $q(0)\cdots q(t)=v$. Note that due to the definition of $Y^{\Sigma}$, 
  the value of $Y^{\Sigma}(u,q)(t)$ does not depend on the choice of $q(l)$, $l > t$.
  Hence, $Y(u,v)$ is well-defined. Define the random variable
  \[ \chi(\theta=v)(\omega)=\left\{\begin{array}{rl}
                      1 & \theta(0)(\omega) \cdots \theta(t)(\omega)=v  \\
                      0 & \mbox{otherwise}.
                     \end{array}\right.
  \]
  It is then easy to see that the output process $\widetilde{y}_t$ of $\Sigma_{\mathrm{st}}$ satisfies
  \[ \widetilde{y}(t) = \sum_{v \in Q^{+}, |v| = t} \frac{1}{(\sqrt{p})^t}Y(u,v)\chi(\theta=v).
  \]
  From this, by noticing that $P(\chi(\theta=v))=p^{t}$ it follows that 
  \[
     E[\widetilde{y}^T(t)\widetilde{y}(t)]=\sum_{v \in Q^{+}, |v|} ||Y(u,v)||_{2}^{2} \ge ||Y(u,v)||^{2}.
  \]
  Hence, for any $(u,q) \in \mathcal{U} \times \mathcal{Q}$, 
  \(
     E[\widetilde{y}^T(t)\widetilde{y}(t)] \ge ||Y^{\Sigma}(u,q)(t)||_{2}^{2}
  \) and thus
  \[
     ||Y^{\Sigma}(u,v)||_{2}^{2} \le \sum_{t=0}^{\infty} E[\widetilde{y}^T(t)\widetilde{y}(t)] \le \gamma^{2} ||u||_{2}^{2},
   \]
   from which the statement follows.  
\end{proof}


\end{document}